\documentclass[12pt,a4paper,oneside]{article}
\usepackage[utf8]{inputenc}
\usepackage{amsmath}
\usepackage{amsfonts}
\usepackage{amssymb}
\usepackage[symbol]{footmisc}
\usepackage{stmaryrd}
\usepackage{graphicx}
\usepackage[english]{babel}
\usepackage{color}
\usepackage{tikz}
\usepackage{mathrsfs}
\usepackage{fancyhdr}
\usepackage[T1]{fontenc}
\usepackage{nicefrac}
\usepackage{dsfont}
\usepackage{amsthm}
\usepackage{soul}
\usepackage{csquotes}
\usepackage[colorlinks=true,linkcolor=blue]{hyperref}
\makeatletter
\renewcommand*{\eqref}[1]{%
  \hyperref[{#1}]{\textup{\tagform@{\ref*{#1}}}}%
}

\usetikzlibrary{calc,backgrounds,fit}
\usetikzlibrary{external}
\usetikzlibrary{patterns}
\usetikzlibrary{decorations.pathreplacing,decorations.markings,arrows}
\usetikzlibrary{shapes,intersections,positioning}
\usetikzlibrary{arrows}

\definecolor{kerstin}{RGB}{25,180,60}
\definecolor{dim}{RGB}{50,175,31}
\definecolor{flu}{RGB}{114,41,196}

\newcommand\centre[4][below]{\node (#3) at #2 [circle,minimum size=0.5em,inner sep=0pt,thin,fill,solid] {}; \node [#1=0.1em] at (#3) {#4};}

\newcounter{cst}
\newcommand{\ctel}[1]{K_{\refstepcounter{cst}\label{#1}\thecst}}
\newcommand{\cter}[1]{K_{\ref*{#1}}} 

\newcommand{\re}{\mathbb{R}}
\newcommand{\R}{\mathbb{R}}
\newcommand{\na}{\mathbb{N}}
\newcommand{\E}{\mathbb{E}}

\newcommand{\diver}{\operatorname{div}_x}

\newcommand{\erwb}{\mathbb{E}\left[}
\newcommand{\erwe}{\right]}
\newcommand{\halbe}{\frac{1}{2}}

\newcommand{\Tau}{\mathcal{T}}
\newcommand{\edges}{\mathcal{E}}
\newcommand{\dkl}{d_{K|L}}
\newcommand{\edgesint}{\mathcal{E}_{\operatorname{int}}}
\newcommand{\edgesext}{\mathcal{E}_{\operatorname{ext}}}
\newcommand{\uhnl}{u_{h,N}^l}
\newcommand{\uhnr}{u_{h,N}^r}

\newcommand{\lzlambda}{{L^2(\Lambda)}}

\newcommand{\whnl}{w_{h,N}^l}
\newcommand{\whnr}{w_{h,N}^r}

\newcommand{\erww}[1]{\mathbb{E}\left[{#1}\right]}
\newcommand{\reg}{\operatorname{reg}(\Tau)}
\newcommand{\upe}{u_{p,\epsilon}}
\newcommand{\upenk}{u_{p,\epsilon}^{n,K}}

\newcommand{\upenpk}{u_{p,\epsilon}^{n+1,K}}
\newcommand{\upenpl}{u_{p,\epsilon}^{n+1,L}}

\newcommand{\tn}{t_n}
\newcommand{\tnp}{t_{n+1}}
\newcommand{\dlt}{\Delta t}
\newcommand{\emct}{e^{-ct}}
\newcommand{\emcs}{e^{-cs}}
\newcommand{\unk}{u^n_K}
\newcommand{\unpk}{u^{n+1}_K}
\newcommand{\unps}{u^{n+1}_\sigma}
\newcommand{\ukpk}{u^{k+1}_K}
\newcommand{\ukpl}{u^{k+1}_L}
\newcommand{\unpl}{u^{n+1}_L}

\newcommand{\ve}{\mathbf{v}}
\newcommand{\ms}{m_{\sigma}}
\newcommand{\nks}{\mathbf{ n}_{K,\sigma}}
\newcommand{\di}{\displaystyle}
\newcommand{\vksnp}{v^{n+1}_{K,\sigma}}
\newcommand{\vkskp}{v^{k+1}_{K,\sigma}}

\newtheorem{defi}{Definition}[section]

\newtheorem{lem}[defi]{Lemma}
\newtheorem{teo}[defi]{Theorem}
\newtheorem{prop}[defi]{Proposition}

\theoremstyle{remark}
\newtheorem{remark}[defi]{Remark}

\numberwithin{equation}{section}

\usepackage[normalem]{ulem} % for strikethrough (\sout)
\usepackage{cancel}

\title{On a finite-volume approximation of a diffusion-convection equation with a multiplicative stochastic force}
\author{Caroline Bauzet\footnotemark[1], \and Kerstin Schmitz\footnotemark[2], \and Aleksandra Zimmermann\footnotemark[2]}
\date{\today}

\begin{document}
\maketitle

\begin{abstract}We address an original approach for the convergence analysis of a finite-volume scheme for the approximation of a stochastic diffusion-convection equation with multiplicative noise in a bounded domain of $\R^d$ (with $d=2$ or $3$) and with homogeneous Neumann boundary conditions. The idea behind our approach is to avoid using the stochastic compactness method. We study a numerical scheme that is semi-implicit in time and in which the convection and the diffusion terms are respectively approximated by means of an upwind scheme and the so called two-point flux approximation scheme (TPFA).
By adapting well-known methods for the time discretization of stochastic PDEs and combining them with deterministic techniques applied to spatial discretization, we show strong convergence of our scheme towards the unique variational solution of the continuous problem in $L^p\left(0,T;L^2(\Omega;L^2(\Lambda))\right)$, for any finite $p\geq 1$.
\quad\\

\noindent\textbf{Keywords:} Stochastic non-linear parabolic equation $\bullet$ Multiplicative Lipschitz noise $\bullet$ Finite-volume method $\bullet$ Upwind scheme $\bullet$ Diffusion-convection equation $\bullet$ Variational approach $\bullet$ Convergence analysis.\\
\quad\\
\textbf{Mathematics Subject Classification (2020):} 60H15 $\bullet$ 35K05 $\bullet$ 65M08.
\end{abstract}
\footnotetext[1]{Aix Marseille Univ, CNRS, Centrale Marseille, LMA UMR 7031, Marseille, France, caroline.bauzet@univ-amu.fr}
\footnotetext[2]{TU Clausthal, Institut f\"ur Mathematik, Clausthal-Zellerfeld, Germany, kerstin.schmitz@tu-clausthal.de,
aleksandra.zimmermann@tu-clausthal.de}

\section{Introduction}
Let $\Lambda$ be a bounded, open, connected, and polygonal subset of $\R^d$ (with $d=2$ or $3$).
Moreover let $(\Omega,\mathcal{A},\mathds{P})$ be a probability space endowed with a right-continuous, complete filtration $(\mathcal{F}_t)_{t\geq 0}$ and let $(W(t))_{t\geq 0}$ be a standard, one-dimensional Brownian motion with respect to $(\mathcal{F}_t)_{t\geq 0}$ on $(\Omega,\mathcal{A},\mathds{P})$.\\
For $T>0$, we consider the following non-linear parabolic problem forced by a multiplicative stochastic noise:
\begin{align}\label{equation}
\begin{aligned}
du-\Delta u\,dt + \diver (\mathbf{v}u)\, dt &=g(u)\,dW(t)+\beta(u)\,dt, &&\text{ in }\Omega\times(0,T)\times\Lambda;\\
u(0,\cdot)&=u_0, &&\text{ in } \Omega\times\Lambda;\\
\nabla u\cdot \mathbf{n}&=0, &&\text{ on }\Omega\times(0,T)\times\partial\Lambda;
\end{aligned}
\end{align}
where $\diver$ is the divergence operator with respect to the space variable and $\mathbf{n}$ denotes the unit normal vector to $\partial\Lambda$ outward to $\Lambda$.
We assume the following hypotheses on the data:
\begin{itemize}
\item[$\mathscr{A}_1$:] $u_0\in L^2(\Omega;L^2(\Lambda))$ is $\mathcal{F}_0$-measurable.
\item[$\mathscr{A}_2$:] $g:\re\rightarrow\re$ is a Lipschitz-continuous function.
\item[$\mathscr{A}_3$:] $\beta : \R\rightarrow \R$ is a Lipschitz-continuous function with $\beta(0)=0$.
\item[$\mathscr{A}_4$:] $\mathbf{v}\in \mathscr{C}^1([0,T]\times \overline{\Lambda}; \R^d)$,
 $\diver(\mathbf{v}(t,x))=0$ for all $(t,x)\in [0,T]\times \Lambda$ and $\mathbf{v}(t,x)\cdot\mathbf{n}(x)=0$ for all $(t,x)\in [0,T]\times \partial\Lambda$.
\end{itemize}

\subsection{Notations}
Let us introduce some notations and make precise the functional setting.
\begin{itemize}
\item[$\bullet$] $|x|$ denotes the euclidean norm of $x$ in $\R^d$ and $x\cdot y$ the usual scalar product of $x$ and $y$ in $\R^d$.
\item By abusing the previous notation, $|\Lambda|$ denotes the $d$-dimensional Lebesgue measure of $\Lambda$.
\item[$\bullet$] For $p\in\{1,d,d+1\}$, $\|\cdot\|_\infty$ denotes the $L^\infty(\re^p)$-norm.
\item[$\bullet$] $L_{\beta}\geq 0$ the Lipschitz constant of $\beta$.
\item[$\bullet$] $L_g\geq 0$ the Lipschitz constant of $g$.
\item[$\bullet$] $C_{L_g}\geq 0$ a constant only depending on $L_g$ and $g(0)$, satisfying for all $r\in\mathbb{R}$ 
\begin{equation}\label{H3}
|g(r)|^2\leq C_{L_g}(1+|r|^2).
\end{equation}
\item[$\bullet$] $\E[\cdot]$ denotes the expectation, \textit{i.e.} the integral over $\Omega$ with respect to the probability measure $\mathds{P}$.
\item[$\bullet$] For a given separable Banach space $X$, we denote by $L^2_{\mathcal{P}_T}\big(\Omega\times(0,T);X\big)$ the space of the predictable $X$-valued processes (\cite{DPZ14} p.94 or \cite{PrevotRockner} p.27). This space is the space $L^2\big(\Omega\times(0,T);X\big)$ for the product measure $d\mathds{P}\otimes dt$ on the predictable $\sigma$-field $\mathcal{P}_T$ (\textit{i.e.} the $\sigma$-field generated by the sets $\mathcal{F}_0\times \{0\}$ and the rectangles $A\times (s,t]$, for any $s,t\in[0,T]$ with $s\leq t$ and $A\in \mathcal{F}_s$).\\ 
For $X=L^2(\Lambda)$, one has $L^2_{\mathcal{P}_T}\big(\Omega\times(0,T);L^2(\Lambda)\big)\subset L^2\big(\Omega\times(0,T);L^2(\Lambda)\big)$.
\end{itemize}
\begin{remark} Note that the existence of the constant $C_{L_g}$ is given by Assumption $\mathscr{A}_2$. It allows us to apply our scheme also for square integrable, additive noise with appropriate measurability assumptions.
\end{remark}

\subsection{Concept of solution and main result}
The theoretical framework associated with Problem~\eqref{equation} is well established in the literature. Indeed, we can find many existence and uniqueness results for various concepts of solutions associated with this problem such as mild solutions, variational solutions, pathwise solutions and weak solutions, see, e.g., \cite{DPZ14,LR,PardouxBook}. In the present paper we will be interested in the concept of solution as defined below, which we will call a variational solution: 
\begin{defi}\label{solution} A stochastic process $u$ in $L^2_{\mathcal{P}_T}\big(\Omega\times(0,T);H^1(\Lambda)\big)$ is 
a variational solution to Problem \eqref{equation} if it belongs to  $L^2(\Omega;\mathscr{C}([0,T];L^2(\Lambda)))$
and satisfies, for all $t\in[0,T]$,
\begin{align*}
u(t)-u_0-\int_0^t \Delta u(s)\,ds+\int_0^t \diver(\mathbf{v}(s,\cdot)u(s))\, ds=\int_0^t g(u(s))\,dW(s)+\int_0^t \beta(u(s))\,ds
\end{align*}
in $L^2(\Lambda)$ and $\mathds{P}$-a.s. in $\Omega$, where $\Delta$ denotes the Laplace operator on $H^1(\Lambda)$ associated with the formal Neumann boundary condition.
\end{defi}
\begin{remark}\label{240202_01}
\textit{A priori}, we have the predictability of $u$ with values in $L^2(\Lambda)$. It is a direct consequence of, e.g., \cite[Corollary 1.1.8]{HNVW16}
that we may \textit{a posteriori} conclude $u\in L^2_{\mathcal{P}_T}\big(\Omega\times(0,T);H^1(\Lambda)\big)$.
\end{remark}

Existence, uniqueness and regularity of this variational solution is well-known in the literature, see, e.g., \cite{KryRoz81,LR,PardouxThese}.
The main result of this paper is to propose a finite-volume scheme for the approximation of such a variational solution and to show its stochastically strong convergence by passing to the limit with respect to the time and space discretization parameters, as stated in the theorem below: 
\begin{teo} \label{mainresult}
Assume that hypotheses $\mathscr{A}_1$ to $\mathscr{A}_4$ hold. 
Let $(\Tau_m)_{m\in \mathbb{N}}$ be a sequence of admissible finite-volume meshes of $\Lambda$ in the sense of Definition \ref{defmesh} such that the mesh size $h_m$ tends to $0$ and let $(N_m)_{m\in \mathbb{N}}\subset \mathbb{N}^{\star}$ be a sequence of positive numbers which tends to infinity.
For a fixed $m\in\mathbb{N}$, let $u^r_{h_m,N_m}$ and $u^l_{h_m,N_m}$ be respectively the right and left in time finite-volume approximations defined by \eqref{eq:notation_wh}, \eqref{eq:def_u0}-\eqref{equationapprox} with $\Tau =\Tau_m$ and $N=N_m$. Then $(u^r_{h_m,N_m})_{m\in \mathbb{N}}$ and $(u^l_{h_m,N_m})_{m\in \mathbb{N}}$ converge strongly in $L^p\left(0,T;L^2(\Omega;L^2(\Lambda))\right)$, for any finite $p\geq 1$,
to the variational solution of Problem \eqref{equation} in the sense of Definition~\ref{solution}.
\end{teo}

\begin{remark}As a consequence of the isomorphism between $L^2(\Omega\times (0,T) ;L^2(\Lambda))$ and $L^2(\Omega ;L^2(0,T ;L^2(\Lambda))$ and $L^2(0,T ;L^2(\Omega ;L^2(\Lambda))$, see, e.g., \cite[Corollary~1.2.23 and Proposition~1.2.24]{HNVW16},
from the convergence result of Theorem \ref{mainresult} it follows that the convergence also holds in $L^2(\Omega;L^2(0,T;L^2(\Lambda)))$.
\end{remark}

\subsection{State of the art}
As mentioned in a previous paper \cite{BNSZ22} in collaboration with {\sc F. Nabet}, the study of numerical schemes for stochastic partial differential equations (SPDEs) has been a very fashionable subject in recent decades and for this reason, an extensive literature on this topic is available. We refer the interested reader to \cite{ACQS20,DP09,OPW20} for a general overview and associated references.\\

If we focus on the theoretical study of parabolic SPDEs with a non-linear first order operator, note that the variational techniques developed in \cite{PardouxThese,KryRoz81,LR} can be applied whereas the semigroup approach is not available, and therefore the use of mild solutions is out of range. Note that for the theoretical study of \eqref{equation} in $\R^d$ ($d\geq 1$) instead of a bounded domain, we can refer to \cite{V08} or the appendix of \cite{BVW12}.\\

Concerning the numerical analysis of these variational solutions, it is clear that in the past the use of finite-element methods has been favored and extensively employed (we refer to \cite{Prohl,BHL21} for a thorough exposition of existing papers).
In recent years, one is more and more interested in numerical approximations which preserve the specific structure of the underlying equations. In this way, important physical properties are reflected at the numerical level and thus improve the robustness and the stability of numerical methods. As pointed out in \cite{Droniou14}, amongst the numerous families of numerical methods, e.g., finite difference, finite element, discontinuous Galerkin, Gradient Discretization Method (GDM), finite-volume schemes are favourable methods for applications in which the conservation of quantities is important. For finite-volume schemes, balance and local conservativity are the leading principles in the construction of numerical fluxes. The TPFA scheme is a cell-centered finite-volume scheme which is particularly cheap to implement, since its matrices are very sparse.\\
There has been growing interest in the use of volume-finite schemes for the spatial discretization of stochastic PDEs. Results have been firstly derived for the approximation by monotone schemes of hyperbolic problems perturbed by multiplicative noise, known as stochastic first order scalar conservation laws in the literature.
Let us cite in the chronological order the following contributions: \cite{BCG162,BCG161,BCG17,FGH18,M18,DV18,BCC20,DV20}. Then, the parabolic case has been more recently investigated in collaboration with {\sc F. Nabet} in \cite{BN20} and \cite{BNSZ22}.
We proved the convergence of a TPFA scheme for the stochastic heat equation with linear and non-linear multiplicative noise, respectively. By adapting the stochastic compactness method based on Skorokhod's representation theorem, we were able to obtain the convergence of our approximation in $L^p(\Omega;L^2(0,T;L^2(\Lambda)))$ for any $p\in [1,2)$, towards the unique variational solution of the associated problem.
Let us additionally mention the recent work of \cite{DGL22}, where the authors investigated the convergence of a large class of numerical schemes for the stochastic transient Leray-Lions equation with multiplicative noise. They used a generic GDM framework which covers various schemes, including finite elements, some finite-volume methods, discontinuous Galerkin, mass-lumped finite elements but does not include, strictly speaking, the TPFA scheme. Using the same framework in a recent preprint \cite{DKL23}, the authors showed a result of convergence of these GDM schemes to a martingale solution of the Stefan's problem perturbed by a multiplicative noise.

\subsection{Originality of the study and outline of the paper}
In our present work, our aim is to fill the gap left by previous authors and to propose a convergence result for a space-time discretization of finite-volume type for the diffusion-convection equation \eqref{equation} with a non-linear source term, a linear, first-order convection term, forced by a multiplicative stochastic noise, and subject to homogeneous Neumann boundary conditions. The added value comparing to existing results is fourfold: 
\begin{itemize}
\item Firstly, the taking into account of a convection term $\diver(\ve u)$ is very interesting from a modeling point of view and paves the way towards many extensions such as the consideration of a non-linear flux term of the type $\mathbf{v}f(u)$ in combination with a porous medium operator of the type $\Delta\varphi(u)$, with $f,\varphi:\R\rightarrow \R$ Lispchitz-continuous, and $\varphi$ nondecreasing.
\item Secondly, the fact that we avoid (for the passage to the limit in the non-linear terms) the use of the stochastic compactness method which involves technical tools from the stochastic framework (such as the theorem of Prokhorov, Skorokhod's representation theorem, the concept of martingale solutions, and the Gy\"ongy-Krylov argument of pathwise uniqueness) as we did in our previous work \cite{BNSZ22} for the stochastic heat equation (which corresponds to \eqref{equation} with $\mathbf{ v}=\mathbf{ 0}$ and $\beta=0$). This allows us to propose a more general convergence result which is at the same time accessible without deeper knowledge of stochastic analysis.
\item Thirdly, we have weakened the hypothesis $u_0\in L^2(\Omega, H^1(\Lambda))$ from \cite{BNSZ22} by assuming that $u_0$ is an element of $L^2(\Omega, L^2(\Lambda))$. This is mainly due to the fact that the method used in the present paper requires drastically fewer stability estimates on approximations than in \cite{BNSZ22}.
\item Fourthly, the extension of the convergence analysis to the dimensional case $d=3$ comparing to \cite{BNSZ22}, where we restricted ourselves to $d=2$.
\end{itemize}

The main difficulty of the present study is to choose suitable tools of the finite-volume framework compatible with the stochastic one and the restrictions brought by the multiplicative noise. Particularly, we will see that the main challenge will be the identification of weak limits coming from the discretization of the non-linear terms $g(u)$ and $\beta(u)$. 
\\

Our contribution is organized as follows. Section \ref{sectiontwo} is devoted to the introduction of the finite-volume framework: definition of the finite-volume mesh employed for the discretization of $\Lambda$, associated notations, definition of discrete norms and construction of the right and left finite-volume approximations denoted respectively by $(u^r_{h,N})_{h,N}$ and $(u^l_{h,N})_{h,N}$. In Section \ref{estimates}, we will derive stability estimates in suitable functional spaces satisfied by the sequences $(u^r_{h,N})_{h,N}$, $(u^l_{h,N})_{h,N}$, $(g(u^l_{h,N}))_{h,N}$ and $(\beta(u^r_{h,N}))_{h,N}$. These estimates will allow us in Section \ref{ConvFVscheme} to extract weakly converging subsequences towards elements of $L^2_{\mathcal{P}_T}\big(\Omega\times(0,T);L^2(\Lambda)\big)$, denoted by $u$, $g_u$ and $\beta_u$. The remaining part will be devoted to the passage to the limit in the numerical scheme and to the identification of $g_u$ and $\beta_u$ as $g(u)$ and $\beta(u)$, respectively. This will be achieved through the following five steps, which combine in an original way techniques commonly used for the time discretization of stochastic PDEs on one hand, and space approximation of deterministic PDEs on the other hand.
\begin{description}
\item[Step 1:] Proving that the joint weak limit $u$ of $(u^r_{h,N})_{h,N}$ and $(u^l_{h,N})_{h,N}$ is an $L^2(\Lambda)$-valued It\^o stochastic process of the form 
$$du+\big(\diver(\mathbf{v}u)-\Delta u-\beta_u\big)dt=g_udW.$$ 
\item[Step 2:] Deriving thanks to It\^o's formula a stochastic energy equality satisfied by the weak limit $u$, by employing in particular an exponential weighted in time norm with a parameter $c>0$. 
\item[Step 3:] This provides in particular a crucial inequality between the $L^2(\Lambda)$-norm of $\nabla u$ and the discrete $H^1(\Lambda)$-seminorm of the sequence $(u^r_{h,N})_{h,N}$.
\item[Step 4:] Obtaining a discrete stochastic energy inequality (using again an exponential weighted in time norm with a parameter $c>0$) satisfied by the quadruple $$\big(u^r_{h,N}, u^l_{h,N}, g(u^l_{h,N}), \beta(u^r_{h,N})\big)_{h,N},$$ as a discrete counterpart to Step 2.
\item[Step 5:] Combining the continuous and discrete energy estimates obtained in Step 2 and Step 4 with the key result derived in Step 3 and making a clever choice of the weight parameter $c>0$ will allow us to do the identifications $g_u=g(u)$ and $\beta_u=\beta(u)$, and therefore to complete the proof of Theorem \ref{mainresult}.
\end{description}

\section{The finite-volume framework}\label{sectiontwo}
The following subsections \ref{mesh}, \ref{discretenotation}, \ref{dnadg} contain all the definitions and notations related to finite-volume framework. 

\subsection{Admissible finite-volume meshes and notations}\label{mesh}
In order to perform a finite-volume approximation of the variational solution of Problem~\eqref{equation} on $[0,T]\times \Lambda$ we need first of all to set a choice for the temporal and spatial discretization. For the time discretization, let $N\in \mathbb{N}^{\star}$ be given. We define the fixed time step $\Delta t=\frac{T}{N}$ and divide the interval $[0,T]$ in $0=t_0<t_1<...<t_N=T$ equidistantly with $\tn=n \Delta t$ for all $n\in \{0, ..., N-1\}$.
For the space discretization, we refer to \cite{gal} and consider finite-volume admissible meshes in the sense of the following definition.

\begin{defi}\label{defmesh} (Admissible finite-volume meshes) An admissible finite-volume mesh of $\Lambda$, denoted by $\Tau$, is given by a family of \enquote{control volumes}, which are open polygonal convex subsets of $\Lambda$, a family of subsets of $\overline{\Lambda}$ contained in hyperplanes of $\R^d$, denoted by $\mathcal{E}$ (these are the edges for $d=2$ or sides for $d=3$ of the control volumes), with strictly positive $(d-1)$-dimensional Lebesgue measure, and a family of points of $\Lambda$ denoted by $\mathcal{P}$ satisfying the following properties\footnotemark[2]\footnotetext[2]{In fact, we shall denote, somewhat incorrectly, by $\mathcal{T}$ the family of control volumes.}
\begin{itemize}
\item $\overline{\Lambda}=\bigcup_{K\in\Tau}\overline{K}$.
\item For any $K\in \Tau$, there exists a subset $\mathcal{E}_K$ of $\mathcal{E}$ such that $\partial K=\overline{K}\setminus K=\bigcup_{\sigma\in\mathcal{E}_K}\overline{\sigma}$. Furthermore, $\mathcal{E}=\bigcup_{K\in \Tau}\mathcal{E}_K$.
\item For any $K,L\in\Tau$, with $K\neq L$ either the $(d-1)$-dimensional Lebesgue measure of $\overline{K}\cap \overline{L}$ is $0$ or $\overline{K}\cap \overline{L}=\overline{\sigma}$, for some $\sigma\in \mathcal{E}$, which will then be denoted by $K|L$.
\item The family $\mathcal{P}=(x_K)_{K\in \Tau}$ is such that for any $K\in \Tau$, $x_K\in \overline{K}$ and, if $\sigma=K|L$, it is assumed that $x_K\neq x_L$, and that the straight line going through $x_K$ and $x_L$ is orthogonal to $K|L$.
\end{itemize}
\end{defi}

Once an admissible finite-volume mesh $\Tau$ of $\Lambda$ is fixed, we will use the following notations.\\
\quad\\
\textbf{Notations.}
\begin{itemize}
\item $h=\operatorname{size}(\Tau)=\sup\{\operatorname{diam}(K): K\in\Tau\}$, the mesh size.
\item $d_h\in\mathbb{N}$, the number of control volumes $K\in\Tau$ with $h=\operatorname{size}(\Tau)$.
\item $\mathcal{E}_{\operatorname{int}}:=\{\sigma\in\mathcal{E}:\sigma\nsubseteq \partial\Lambda\}$  and $\mathcal{E}_{\operatorname{ext}}:=\{\sigma\in\mathcal{E}:\sigma\subseteq \partial\Lambda\}$.

\item For any $K\in\Tau$, $m_K$ denotes the $d$-dimensional Lebesgue measure  of $K$ (it is the area when $d=2$ and the volume when $d=3$), and for any $\sigma\in \mathcal{E}_{\operatorname{int}}$, $m_\sigma$ denotes the $(d-1)$-dimensional Lebesgue measure  of $\sigma$.
\item For any $K\in\Tau$, $\mathbf{n}_K$ denotes the unit normal vector to $\partial K$ outward to $K$, and for any $\sigma \in \mathcal{E}_K$, the unit vector on the edge $\sigma$ pointing out of $K$ is denoted by $\mathbf{n}_{K,\sigma}$.

\item  For any $\sigma=K|L\in\edgesint$, $d_{K|L}$ denotes the Euclidean distance between $x_K$ and $x_L$ (which is positive).

\end{itemize}

\begin{figure}[htbp!]
\centering
\begin{tikzpicture}[scale=2]

  \clip (-1.2,-0.6) rectangle (1.8,1.3);

  \node[rectangle,fill] (A) at (-1,0.6) {};
  \node[rectangle,fill] (B) at (0,1.2) {};
  \node[rectangle,fill] (C) at (0,-0.2) {};
  \node[rectangle,fill] (D) at (1.5,0.3) {};

  \centre[above right]{(-0.6,0.5)}{xK}{$x_K$};
  \centre[above left]{(0.9,0.5)}{xL}{$x_L$};
  
  \draw[thick] (B)--(C) node [pos=0.7,right] {$\sigma=$\small{$K|L$}};

  \draw[thin,opacity=0.5] (A) -- (B) -- (C) -- (A) ;
  \draw[thin,opacity=0.5] (D) -- (B) -- (C) -- (D);

  \draw[dashed] (xK) -- (xL);
  
  \coordinate (KK) at ($(xK)!0.65!-90:(xL)$);
  \coordinate (LL) at ($(xL)!0.65!90:(xK)$);

  \draw[dotted,thin] (xK) -- (KK);
  \draw[dotted,thin] (xL) -- (LL);

  \draw[|<->|] (KK) -- (LL) node [midway,fill=white,sloped] {$\dkl$};
 
  \coordinate (KAB) at ($(A)!(xK)!(B)$);
  \coordinate (KAC) at ($(A)!(xK)!(C)$);

  \coordinate (LDB) at ($(D)!(xL)!(B)$);
  \coordinate (LDC) at ($(D)!(xL)!(C)$);

  \draw[dashed] (xK) -- ($(xK)!3!(KAB)$);
  \draw[dashed] (xK) -- ($(xK)!3!(KAC)$);

  \draw[dashed] (xL) -- ($(xL)!3!(LDB)$);
  \draw[dashed] (xL) -- ($(xL)!3!(LDC)$);

  \begin{scope}[on background layer]   
    \draw (0,0.5) rectangle ++ (0.1,-0.1);
  \end{scope}
  
  \coordinate (nkl) at ($(B)!0.35!(C)$);
  \draw[->,>=latex] (nkl) -- ($(nkl)!0.3cm!90:(C)$) node[above] {$\mathbf{n}_{K,\sigma}$};

\end{tikzpicture}
\caption{Notations of the mesh $\mathcal T$ associated with $\Lambda$\label{fig:notation_mesh} for $d=2$}
\end{figure}
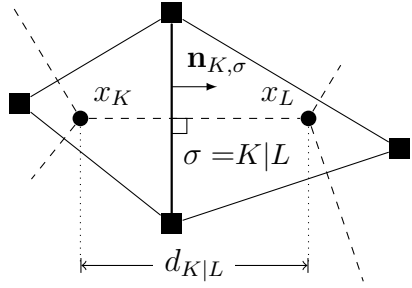

Using these notations, and by denoting additionally $\mathcal N$ the  maximum of edges incident to any vertex, and for any $\sigma\in\mathcal{E}_K$, and any $x_K\in\mathcal{P}$, $d(x_K,\sigma)$ the Euclidean distance between $x_K$ and $\sigma$, we introduce a positive number 
\begin{eqnarray}\label{mrp}
\reg=\max\left(\mathcal N,\max_{\scriptscriptstyle K \in\Tau \atop \sigma\in\mathcal{E}_K} \frac{\operatorname{diam}(K)}{d(x_K,\sigma)}\right)
\end{eqnarray}
 that measures the regularity of a given mesh and is useful to perform the convergence analysis of our finite-volume scheme.
This number should be uniformly bounded by a constant not depending on the mesh size $h$ for the convergence results to hold. We have in particular for any $ K,L\in \Tau$,  
\begin{equation}\label{hoverdkl}
\frac{h}{\dkl}\leq \reg.
\end{equation}

\subsection{Discrete unknowns and piecewise constant functions }\label{discretenotation}

From now on and unless otherwise specified, we consider $N\in \mathbb{N}^{\star}$, $\Delta t=\frac{T}{N}$ and $\Tau$ an admissible finite-volume mesh of $\Lambda$ in the sense of Definition \ref{defmesh} with a mesh size $h$. For $n\in\{0, ..., N-1\}$ given, the idea of a finite-volume scheme for the approximation of Problem \eqref{equation} is to associate to each control volume $K\in\Tau$ and time $t_n$ a discrete unknown value denoted by $u^n_K\in \mathbb{R}$, expected to be an approximation of $u(t_n,x_K)$, where $u$ is the variational solution of \eqref{equation}. Before presenting the numerical scheme satisfied by the discrete unknowns $\{u^n_K, K\in\Tau, n\in\{0, ..., N-1\}\}$, let us introduce some general notations.\\

For any arbitrary vector $(w_K^n)_{K\in\Tau}\in\re^{d_h}$ we can define the piecewise constant function $w_h^n:\Lambda\rightarrow \mathbb{R}$ by
\[
w_h^n(x):=\sum_{K\in\Tau} w^n_K \mathds{1}_K(x),\ \forall x\in \Lambda.
\]
Note that since the mesh $\Tau$ is fixed, by the continuous mapping defined from $ \mathbb{R}^{d_h}$ to $L^2(\Lambda)$ by
\[(w^n_K)_{{K\in\Tau}} 
\mapsto \sum_{K\in\Tau}\mathds{1}_K w^n_K, \]
the space $\re^{d_h}$ can be considered as a finite-dimensional subspace of $L^2(\Lambda)$ and we may naturally identify the function and the vector
\[
w^n_h\equiv(w^n_K)_{K\in\Tau}\in \mathbb{R}^{d_h}.
\]
Then, knowing for all $n \in\{0,\ldots,N\}$ the function $w_h^n$, we can define the following piecewise constant functions in time and space $\whnr, \whnl :[0,T]\times \Lambda\rightarrow \mathbb{R}$ by
\begin{equation}
 \label{eq:notation_wh}
 \begin{aligned}  
  \whnr(t,x):=&\sum_{n=0}^{N-1} w_h^{n+1}(x)\mathds{1}_{[t_n,t_{n+1})}(t)\text{ if }t\in[0,T)
\text{ and } \whnr(T,x):=w_h^N(x),\\
\whnl(t,x):=&\sum_{n=0}^{N-1} w_h^n(x)\mathds{1}_{[t_n,t_{n+1})}(t) \text{ if }t\in(0,T]\text{ and } 
\whnl(0,x):=w_h^0(x).
 \end{aligned}
\end{equation}
\begin{remark}
The superscripts $r$ and $l$ in \eqref{eq:notation_wh} do not refer to the continuity properties of the associated functions (which may be chosen either c\`{a}dl\`{a}g or c\`{a}gl\`{a}d). The difference is that in our case, the finite-volume approximation $u_{h,N}^l$ will be adapted to the filtration $(\mathcal{F}_t)_{t\geq 0}$, whereas $u_{h,N}^r$ won't be.
\end{remark}
As for the piecewise constant function in space, since $\Tau$ and $N$ are fixed, by the continuous mapping defined from $\mathbb{R}^{d_h\times N}$ to $L^2(0,T;L^2(\Lambda))$ by
\[(w_K^n)_{\substack{K\in\Tau \\ n\in\{0,\ldots,N-1\}}}\mapsto\sum_{\substack{K\in\Tau \\ n\in\{0,\ldots,N-1\}}}\mathds{1}_K\mathds{1}_{[t_n,t_{n+1})}w_K^n,\]
the space $\mathbb{R}^{d_h\times N}$ can be considered as a finite-dimensional subspace of $L^2(0,T;L^2(\Lambda))$ and we may naturally identify
\begin{align*}
\whnl&\equiv(w_K^n)_{\substack{K\in\Tau \\ n\in\{0,\ldots,N-1\}}}\in \mathbb{R}^{d_h\times N},\\
\whnr&\equiv(w_K^{n+1})_{\substack{K\in\Tau \\ n\in\{0,\ldots,N-1\}}}\in \mathbb{R}^{d_h\times N}.
\end{align*}
\begin{remark}Note that in the rest of the paper, when we will consider a time and space function $\alpha:[0,T]\times \Lambda\rightarrow \mathbb{R}$ on all the space $\Lambda$ (respectively the time interval $[0,T]$) at a fixed time $t\in [0,T]$ (respectively at a fixed $x\in \Lambda$) we will omit the space (respectively time) variable in the notations and write $\alpha(t)$ (respectively $\alpha(x)$) instead of $\alpha(t,\cdot)$ (respectively $\alpha(\cdot,x)$).
\end{remark}

\subsection{Discrete norms and weak gradient}\label{dnadg}
Fix $n\in\{0, ..., N-1\}$ and consider for the remainder of this subsection an arbitrary vector $(w^n_K)_{K\in\Tau}\in \mathbb{R}^{d_h}$ and use its natural identification with the piecewise constant function in space $w^n_h\equiv(w^n_K)_{K\in\Tau}$. We introduce in what follows the notions of weak gradient and discrete norms for such a function $w^n_h$.
\begin{defi}[Discrete $L^2$-norm]We define the $L^2$-norm of $w^n_h \in\re^{d_h}$ as follows
$$||w^n_h||_{L^2(\Lambda)}=\left(\sum_{K\in \Tau}m_K |w^n_K|^2\right)^\frac12.$$
\end{defi}

\begin{defi}[Weak gradient]
We define the gradient operator $\nabla^h$ that maps scalar fields $w^n_h\in\re^{d_h}$ into vector fields of $(\re^{d})^{e_h}$ (where $e_h$ is the number of elements of $\mathcal{E}$),
we set $\nabla^h w^n_h =(\nabla_\sigma^h w^n_h)_{\sigma\in\edges}$ with
 \[
  \nabla_\sigma^h w^n_h :=
  \left\{
  \begin{aligned}
   d\frac{w^n_L-w^n_K}{\dkl} \mathbf{n}_{K,\sigma}, \quad
   &\text{ if }\sigma=K|L\in\edgesint ; \\
   \qquad 0, \qquad\qquad &\text{ if } \sigma\in\edgesext.
  \end{aligned}
  \right.
 \]
\end{defi}

\begin{defi}[Discrete $H^1$-seminorm]
We define the $H^1$-seminorm of $w^n_h \in\re^{d_h}$ as follows
 \[
  |w^n_h|_{1,h}:=\left(\sum_{\sigma=K|L\in\edgesint}\frac{m_\sigma}{\dkl}|w^n_K-w^n_L|^2\right)^\halbe.
 \]

\end{defi}

\begin{remark}\label{discrpartint}
If we consider another arbitrary vector $\widetilde w^n_h\equiv(\widetilde{w}^n_K)_{K\in\Tau}\in \mathbb{R}^{d_h}$, by summing over the edges we may rearrange the sum on the left-hand side and get the following rule of "discrete partial integration"
\begin{align}\label{PInt}
\sum_{K\in\Tau}\sum_{\sigma=K|L\in\edges_K\cap\edgesint}\frac{m_\sigma}{\dkl}(w^n_K-w^n_L)\widetilde w^n_K
=\sum_{\sigma=K|L\in\edgesint}\frac{m_\sigma}{\dkl}(w^n_K-w^n_L)(\widetilde w^n_K-\widetilde w^n_L).
\end{align}
\end{remark}
\quad\\
Now, we have all the necessary definitions and notations to present the finite-volume scheme studied in this paper. This is the aim of the next subsection.

\subsection{The finite-volume scheme}

Firstly, we define the vector $u_h^0\equiv (u^0_K)_{K\in\Tau} \in \re^{d_h}$ by the discretization of the initial condition $u_0$ of Problem \eqref{equation} over each control volume:
\begin{align}
\label{eq:def_u0}
u_K^0:=\frac{1}{m_K}\int_K u_0(x)\,dx, \quad \forall K\in \Tau.
\end{align}
The finite-volume scheme we propose reads, for this given initial $\mathcal{F}_0$-measurable random vector $u_h^0\in\re^{d_h}$, as follows: \\
for any $n \in \{0,\dots,N-1\}$, knowing $u_h^n\equiv (u^{n}_K)_{K\in\Tau} \in \re^{d_h}$ we search $u_h^{n+1}\equiv(u^{n+1}_K)_{K\in\Tau}\in\re^{d_h}$ such that, for almost every $\omega\in\Omega$, the vector $u_h^{n+1}$ is a solution to the following random equations 
\begin{align}\label{equationapprox}
\begin{split}
&\frac{m_K}{\Delta t}(u_K^{n+1}-u_K^n)+\sum_{\sigma\in\edgesint\cap\edges_K} m_\sigma \vksnp u^{n+1}_{\sigma} +\sum_{\sigma=K|L\in\edgesint\cap\edges_K}\frac{m_\sigma}{\dkl}(u_K^{n+1}-u_L^{n+1})\\
&=\frac{m_K}{\Delta t}g(u_K^n)(W^{n+1}-W^n)+m_K \beta(\unpk),\quad \forall K\in \Tau,
\end{split}
\end{align}
where, the $(d-1)$-dimensional Lebesgue measure is denoted by $\gamma$,
$$\vksnp=\frac{1}{\dlt \ms}\int_{\tn}^{\tnp}\int_{\sigma}\ve(t,x)\cdot\mathbf{ n}_{K,\sigma}\, d\gamma(x)\,dt,$$ 
and $u^{n+1}_{\sigma}$ denotes the upstream value at time $t_{n+1}$ with respect to $\sigma$ defined as follows: if $\sigma\in\edgesint\cap\edges_K$ is the interface between the control volumes $K$ and $L$ (\textit{i.e.} $\sigma=K|L$), $u^{n+1}_{\sigma}$ is equal to $u^{n+1}_K$ if $\vksnp\geq 0$ and to $u^{n+1}_L$ if $\vksnp< 0$.
Note also that $W^{n+1}-W^n$ denotes the increments of the Brownian motion between $t_{n+1}$ and $t_n$:
$$W^{n+1}-W^n=W(t_{n+1})-W(t_n)\text{ for }n\in\{0,\dots,N-1\}.$$ 

\begin{remark} Note that using the divergence-free property of $\ve$ (\textit{i.e.} $\diver(\mathbf{v}(t,x))=0$ for all $(t,x)\in[0,T]\times \Lambda$), the scheme \eqref{equationapprox} can be rewritten in the following way:
\begin{align}\label{equationapproxbis}
&\frac{m_K}{\Delta t}(u_K^{n+1}-u_K^n)+\sum_{\sigma\in\edgesint\cap\edges_K} m_\sigma \vksnp(u^{n+1}_{\sigma}-u^{n+1}_K)+\sum_{\sigma=K|L\in\edgesint\cap\edges_K}\frac{m_\sigma}{\dkl}(u_K^{n+1}-u_L^{n+1})\nonumber \\
&=\frac{m_K}{\Delta t}g(u_K^n)\left(W^{n+1}-W^n\right)+m_K \beta(\unpk),\quad  \forall K \in \Tau.
\end{align}
Indeed, 
\begin{eqnarray*}
\sum_{\sigma\in\edgesint\cap\edges_K} m_\sigma \vksnp u^{n+1}_K
&=&\sum_{\sigma\in\edgesint\cap\edges_K} \frac{u^{n+1}_K}{\dlt} \int_{\tn}^{\tnp}\int_{\sigma}\ve(t,x)\cdot\mathbf{ n}_{K,\sigma}\,d\gamma(x)\,dt\\
&=&\frac{u^{n+1}_K}{\dlt}\int_{\tn}^{\tnp}\int_{\partial K}\mathbf{v}(t,x)\cdot \mathbf{n}_{K}(x)\,d\gamma(x)\,dt\\
&=&\frac{u^{n+1}_K}{\dlt}\int_{\tn}^{\tnp} \int_{K}\diver(\mathbf{v}(t,x))\,dx\,dt\\
&=&0.
\end{eqnarray*}

Note that using \eqref{equationapproxbis} and the fact that\footnotemark[1] \footnotetext[1]{For any $a\in\R$, $a^+=\max(a,0)$ and $a^-=-\min(a,0)$.} $\vksnp=(\vksnp)^+-(\vksnp)^-$, another equivalent formulation of the scheme \eqref{equationapprox} is given for any $K\in\Tau$ by
\begin{align}\label{equationapproxter}
\begin{split}
&\frac{m_K}{\Delta t}(u_K^{n+1}-u_K^n)+\sum_{\sigma=K|L\in\edgesint\cap\edges_K} m_\sigma (\vksnp)^-\big(u^{n+1}_{K}-u^{n+1}_L\big)\\
&+\sum_{\sigma=K|L\in\edgesint\cap\edges_K}\frac{m_\sigma}{\dkl}(u_K^{n+1}-u_L^{n+1})=\frac{m_K}{\Delta t}g(u_K^n)\left(W^{n+1}-W^n\right)+m_K \beta(\unpk).
\end{split}
\end{align}
\end{remark}

\begin{prop}[Existence of a discrete solution]
\label{210609_prop1}
Assume that hypotheses $\mathscr{A}_1$ to $\mathscr{A}_4$ hold. Let $\Tau$ be an admissible finite-volume mesh of $\Lambda$  in the sense of Definition \ref{defmesh} with a mesh size $h$ and $N\in \mathbb{N}^{\star}$. Then, there exists a unique solution $(u_h^n)_{1\le n \le N} \in (\re^{d_h})^N$ to Problem~\eqref{equationapprox} associated with the initial vector $u^0_h$ defined by~\eqref{eq:def_u0}. Additionally, for any $n\in \{0,\ldots,N\}$, $u_h^n$ is a $\mathcal{F}_{t_n}$-measurable random vector.
\end{prop}

The solution $(u_h^n)_{1\le n \le N} \in (\re^{d_h})^N$ of the scheme \eqref{eq:def_u0}-\eqref{equationapprox} is then used to build the right and left finite-volume approximations $u^r_{h,N}$ and $u^l_{h,N}$ defined by \eqref{eq:notation_wh} for the variational solution $u$ of Problem \eqref{equation}.

\begin{proof} We refer to \cite{BNSZ23} exclusively dedicated to the proof of such an existence and uniqueness result. 
\end{proof}

\section{Stability estimates}\label{estimates}

We will derive in this section several stability estimates satisfied by the discrete solution $(u_h^n)_{1 \le n \le N} \in (\re^{d_h})^N$ of the scheme \eqref{eq:def_u0}-\eqref{equationapprox} given by Proposition \ref{210609_prop1}, and also by the associated right and left finite-volume approximations $(u^r_{h,N})_{h,N}$ and $(u^l_{h,N})_{h,N}$ defined by \eqref{eq:notation_wh}. We start by giving a bound on the discrete initial data, as a direct consequence of the definition of $u_h^0$ and the Cauchy-Schwarz inequality:
\begin{lem}
\label{bound_u0}
Let $u_0$ be a given function satisfying assumption $\mathscr{A}_1$. Then, the associated discrete initial data $u_h^0 \in \re^{d_h}$ defined by~\eqref{eq:def_u0} satisfies $\mathds{P}$-a.s. in $\Omega$,
\begin{equation*}
\|u_h^0\|_{L^2(\Lambda)} \leq \|u_0\|_{L^2(\Lambda)}.
\end{equation*}
\end{lem}
Now, we can give the bounds on the discrete solutions which is one of the key points of the proof of the convergence theorem.

\begin{prop}[Bounds on the discrete solutions]\label{bounds}
There exists a constant $K_0\geq 0$, depending only on $u_0$, $C_{L_g}$, $L_{\beta}$, $|\Lambda|$ and $T$ such that for any $N\in\mathbb{N}^{\star}$ large enough and any $h\in \R^\star_+$
\begin{align*}
&\erwb \|u_h^n\|_{L^2(\Lambda)}^2 \erwe+\erwb\sum_{k=0}^{n-1}\|u_h^{k+1}-u_h^k\|_{L^2(\Lambda)}^2\erwe+\Delta t \sum_{k=0}^{n-1}\erww{|u_h^{k+1}|_{1,h}^2}\leq K_0,\; \forall n\in \{1,\ldots,N\}.
\end{align*}
\end{prop}

\begin{proof}
Set $N\in\mathbb{N}^{\star}$, $h\in \mathbb{R}_+^\star$ and fix $n\in \{1,\ldots,N\}$. For any $k\in \{0,\ldots,n-1\}$, we multiply the numerical scheme \eqref{equationapproxter} with $u_K^{k+1}$, take the expectation, and sum over $K\in\Tau$ to obtain thanks to \eqref{PInt}
\begin{align}\label{implicitscheme}
\begin{split}
&\sum_{K\in\Tau}\frac{m_K}{\Delta t}\erwb(u_K^{k+1}-u_K^k)u_K^{k+1}\erwe
+ \sum_{\sigma=K|L\in\edgesint}\frac{m_\sigma}{\dkl}\erwb|u_K^{k+1}-u_L^{k+1}|^2\erwe\\
&+\sum_{K\in\Tau}\sum_{\sigma=K|L\in\edgesint\cap\edges_K} m_\sigma (\vkskp)^-\erwb\big(\ukpk-\ukpl\big)u_K^{k+1}\erwe\\
=\,&\sum_{K\in\Tau}\frac{m_K}{\Delta t}\erwb g(u_K^k)u_K^{k+1}\left(W^{k+1}-W^k\right)\erwe+\sum_{K\in\Tau}m_K \erwb \beta(\ukpk)\ukpk\erwe.
\end{split}
\end{align}
We consider the terms of \eqref{implicitscheme} separately. Firstly note that
\begin{align}\label{term1}
\sum_{K\in\Tau}\frac{m_K}{\Delta t}\erwb(u_K^{k+1}-u_K^k)u_K^{k+1}\erwe=\halbe\sum_{K\in\Tau}\frac{m_K}{\Delta t}\erwb|u_K^{k+1}|^2-|u_K^k|^2+|u_K^{k+1}-u_K^k|^2\erwe.
\end{align}
Secondly, using the inequality $\forall a,b\in\R, \ \displaystyle b(b-a)\geq \frac{b^2}{2}-\frac{a^2}{2}$ with $b=u_K^{k+1}$ and $a=u_L^{k+1}$, one arrives at 
$$\big(u^{k+1}_{K}-u^{k+1}_L\big)u_K^{k+1}\geq \frac{(u_K^{k+1})^2}{2}-\frac{(u_L^{k+1})^2}{2}.$$
Note that $v_{L,\sigma}^{k+1}=-\vkskp$, thus using the divergence-free property of $\ve$ one gets
\begin{align*}
&\sum_{K\in \Tau}\sum_{\sigma=K|L\in\edgesint\cap\edges_K} m_\sigma (\vkskp)^- \left(\frac{(u_L^{k+1})^2}{2}-\frac{(u_K^{k+1})^2}{2}\right)\\
=&\sum_{K\in \Tau}\sum_{\sigma\in\edgesint\cap\edges_K} m_\sigma \vkskp \frac{(u_K^{k+1})^2}{2}\\
=&\sum_{K\in \Tau}\sum_{\sigma\in\edgesint\cap\edges_K}\frac{(u_K^{k+1})^2}{2\dlt}\int_{\tn}^{\tnp} \int_\sigma \ve(t,x)\cdot\mathbf{n}_{K,\sigma}\,d\gamma(x)\,dt \\
=&\sum_{K\in \Tau} \frac{(u_K^{k+1})^2}{2\dlt}\int_{\tn}^{\tnp} \int_K \diver(\ve(t,x))\,dx\,dt \\
=&0,
\end{align*}
and this leads to
\begin{eqnarray}\label{term2}
\sum_{K\in\Tau}\sum_{\sigma=K|L\in\edgesint\cap\edges_K} m_\sigma (\vkskp)^-\E\left[\big(\ukpk-\ukpl\big)u_K^{k+1}\right]\geq 0.
\end{eqnarray}
Thirdly, since $u_K^k$ and $\left(W^{k+1}-W^k\right)$ are independent one obtains
\begin{eqnarray*}
\sum_{K\in\Tau}\frac{m_K}{\Delta t}\erwb g(u_K^k)u_K^k\left(W^{k+1}-W^k\right)\erwe=0,
\end{eqnarray*}
and so by applying Young's inequality and using the It\^{o} isometry one arrives at
\begin{align}\label{term3}
\begin{split}
&\sum_{K\in\Tau}\frac{m_K}{\Delta t}\erwb g(u_K^k)u_K^{k+1}\left(W^{k+1}-W^k\right)\erwe\\
=\,&\sum_{K\in\Tau}\frac{m_K}{\Delta t}\erwb g(u_K^k)(u_K^{k+1}-u_K^k)\left(W^{k+1}-W^k\right)\erwe\\
\leq\,&\sum_{K\in\Tau}\frac{m_K}{\Delta t}\erwb|g(u_K^k)\left(W^{k+1}-W^k\right)|^2\erwe+\frac{1}{4}\sum_{K\in\Tau}\frac{m_K}{\Delta t}\erwb|u_K^{k+1}-u_K^k|^2\erwe\\
\leq\, &\Delta t\sum_{K\in\Tau}\frac{m_K}{\Delta t}\erwb|g(u_K^k)|^2\erwe+\frac{1}{4}\sum_{K\in\Tau}\frac{m_K}{\Delta t}\erwb|u_K^{k+1}-u_K^k|^2\erwe.
\end{split}
\end{align}
Fourthly, using the Lipschitz property of $\beta$, the following holds
\begin{eqnarray}\label{term4}
\sum_{K\in\Tau}m_K \erwb \beta(\ukpk)\ukpk\erwe&\leq& L_{\beta}\sum_{K\in\Tau}m_K \erwb |\ukpk|^2\erwe.
\end{eqnarray}\noindent Combining \eqref{term1}-\eqref{term2}-\eqref{term3} and \eqref{term4} and multiplying the obtained inequality with $2\Delta t$, one gets
\begin{eqnarray*}
&&\sum_{K\in\Tau}m_K\erwb |u_K^{k+1}|^2-|u_K^k|^2+|u_K^{k+1}-u_K^k|^2\erwe+ 2\dlt\sum_{\sigma=K|L\in\edgesint}\frac{m_\sigma}{\dkl}\erwb|u_K^{k+1}-u_L^{k+1}|^2\erwe\\
&\leq&2\dlt\sum_{K\in\Tau}m_K\erwb|g(u_K^k)|^2\erwe+\frac{1}{2}\sum_{K\in\Tau}m_K\erwb|u_K^{k+1}-u_K^k|^2\erwe+2\dlt L_{\beta}\sum_{K\in\Tau}m_K \erwb |\ukpk|^2\erwe.
\end{eqnarray*}
Then, from \eqref{H3}
\begin{align*}
&(1-2\dlt L_{\beta}) \sum_{K\in\Tau}m_K\erwb |u_K^{k+1}|^2-|u_K^k|^2\erwe
+\frac{1}{2}\sum_{K\in\Tau}m_K\erwb|u_K^{k+1}-u_K^k|^2\erwe\\
&+ 2\dlt\sum_{\sigma=K|L\in\edgesint}\frac{m_\sigma}{\dkl}\erwb|u_K^{k+1}-u_L^{k+1}|^2\erwe\leq 2C_{L_g}\dlt |\Lambda|+ 2\dlt (C_{L_g}+L_{\beta})\sum_{K\in\Tau}m_K \erwb (u^k_K)^2\erwe.
\end{align*}
For $\dlt$ small enough so that $1-2\dlt L_{\beta}\geq \frac{1}{4}$, after summing over $k\in\{0,\dots,n-1\}$, one arrives at
\begin{align}\label{lem1beschr}
\begin{split}
&\frac{1}{4}\erwb\|u_h^n\|_{L^2(\Lambda)}^2-\|u_h^0\|_{L^2(\Lambda)}^2\erwe+\frac{1}{2}\sum_{k=0}^{n-1}\erww{\|u_h^{k+1}-u_h^k\|_{L^2(\Lambda)}^2}+ 2\Delta t\sum_{k=0}^{n-1}\erwb|u_h^{k+1}|_{1,h}^2\erwe\\
&\leq 2C_{L_g}T |\Lambda|+ 2\dlt (C_{L_g}+L_{\beta}) \sum_{k=0}^{n-1}\erwb\|u_h^k\|_{L^2(\Lambda)}^2\erwe.
\end{split}
\end{align}
Then, it follows that
\[
\erwb\|u_h^n\|_{L^2(\Lambda)}^2\erwe\leq\erwb\|u_h^0\|_{L^2(\Lambda)}^2\erwe+8C_{L_g}|\Lambda|T+8\Delta t(C_{L_g}+L_{\beta})\sum_{k=0}^{n-1}\erwb\|u_h^k\|_{L^2(\Lambda)}^2\erwe.
\]
Applying the discrete Gronwall lemma yields 
\begin{align}\label{uhnbound}
\erwb\|u_h^n\|_{L^2(\Lambda)}^2\erwe\leq\left(\erwb\|u_h^0\|_{L^2(\Lambda)}^2\erwe+8C_{L_g}|\Lambda|T\right)e^{8T(C_{L_g}+L_{\beta})}.
\end{align}
From \eqref{uhnbound} and Lemma~\ref{bound_u0} we may conclude that there exists a constant $\Upsilon>0$ such that
\begin{align}\label{uhnboundbis}
\sup_{n\in\{1,\dots,N\}}\erww{\|u_h^n\|_{L^2(\Lambda)}^2}\leq \Upsilon.
\end{align}
Thanks to \eqref{uhnboundbis} and \eqref{H3} one gets that for all $n\in\{1,\ldots N\}$
\begin{align}\label{210819_02}
\Delta t\sum_{k=0}^{n-1}\erww{\|g(u_h^k)\|_{L^2(\Lambda)}^2}\leq C_{L_g}|\Lambda|T+C_{L_g}\Delta t\sum_{k=0}^{n-1}\erww{\|u_h^k\|_{L^2(\Lambda)}^2}\leq C_{L_g}T(\Upsilon+|\Lambda|).
\end{align}
From \eqref{lem1beschr}, Lemma~\ref{bound_u0} and \eqref{uhnboundbis} it now follows that for all $n\in \{1,\ldots, N\}$
\begin{align*}
\begin{split}
&\erwb\|u_h^n\|_{L^2(\Lambda)}^2\erwe+2\sum_{k=0}^{n-1}\erww{\|u_h^{k+1}-u_h^k\|_{L^2(\Lambda)}^2}+ 8\Delta t\sum_{k=0}^{n-1}\erwb|u_h^{k+1}|_{1,h}^2\erwe\\
&\leq  \erwb\|u_0\|_{L^2(\Lambda)}^2\erwe+8C_{L_g}T |\Lambda|+8\Upsilon T(C_{L_g}+L_{\beta}).
\end{split}
\end{align*}
Now, defining $K_0\geq 0$ to be the right-hand side of the above equation, the result follows.
\end{proof}
We are now interested in the bounds on the right and left finite-volume approximations defined by~\eqref{eq:notation_wh}.

\begin{lem}\label{210611_lem01}
The sequences $(\uhnr)_{h,N}$ and $(\uhnl)_{h,N}$ are bounded in $L^2(\Omega;L^2(0,T;\lzlambda))$ independently of the discretization parameters $N\in\mathbb{N}^{\star}$ and $h\in \mathbb{R}_+^\star$. Additionally, $(\uhnl)_{h,N}$ is bounded in $L^2_{\mathcal{P}_T}\big(\Omega\times(0,T);L^2(\Lambda)\big)$.
\end{lem}
\begin{proof} The boundedness of the sequences in $L^2(\Omega;L^2(0,T;\lzlambda))$ is a direct consequence of Proposition~\ref{bounds}. The predictability of $(\uhnl)_{h,N}$ with values in $\lzlambda$ is a consequence of the $\mathcal{F}_{\tn}$-measurability of $\unk$ for all $n\in\{0, ..., N\}$ and all $K\in\Tau$. Indeed, by construction, $(\uhnl)_{h,N}$ is then an elementary process adapted to the filtration $(\mathcal{F}_t)_{t\geq 0}$ and so it is predictable.
\end{proof}

Thanks to Proposition~\ref{bounds} we can also obtain a $L^2(\Omega;L^2(0,T;L^2(\Lambda)))$-bound on the weak gradients of the finite-volume approximations.

\begin{lem}\label{remarkuhnrboundintomega}
There exists a constant $\ctel{K1}\geq 0$ depending only on $u_0$, $C_{L_g}$, $L_\beta$, $|\Lambda|$ and $T$  such that
\begin{align}\label{uhnrboundinotomega}
\int_0^T\erwb|\uhnr(t)|_{1,h}^2\erwe dt\leq \cter{K1}.
\end{align}
\end{lem}
\begin{proof} This estimate follows directly from Proposition~\ref{bounds}.
\end{proof}

\begin{lem}\label{boundguhnlr}
The sequences $(g(\uhnr))_{h,N}$, $(g(\uhnl))_{h,N}$, $(\beta(\uhnr))_{h,N}$, and $(\beta(\uhnl))_{h,N}$ are bounded in $L^2(\Omega;L^2(0,T;\lzlambda))$ independently of the discretization parameters $N\in\mathbb{N}^{\star}$ and $h\in \mathbb{R}_+^{\star}$.
\end{lem}
\begin{proof} It is a direct consequence of the boundedness of the sequences $(u_{h,N}^r)_{h,N}$ and $(u_{h,N}^l)_{h,N}$ in $L^2(\Omega;L^2(0,T;L^2(\Lambda)))$ given by Lemma \ref{210611_lem01} and of the Lipschitz nature of $g$ and $\beta$.
\end{proof}
\begin{lem}\label{boundsnguhnlr}
There exists a constant $\ctel{K3}\geq 0$ depending only on $u_0$, $L_g$, $L_\beta$, $|\Lambda|$ and $T$ such that 
\begin{align}\label{uhnrboundinotomegabis}
\int_0^T\erwb|g(\uhnr(t))|_{1,h}^2\erwe dt\leq \cter{K3}.
\end{align}
\end{lem}
\begin{proof}After noticing that:
\begin{eqnarray*}
\int_0^T\erwb|g(\uhnr(t))|_{1,h}^2\erwe dt& \leq& L_g^2\int_0^T\erwb|\uhnr(t)|_{1,h}^2\erwe dt,
\end{eqnarray*}
the result is immediate thanks to Lemma \ref{remarkuhnrboundintomega}.
\end{proof}

\section{Convergence of the finite-volume scheme}\label{ConvFVscheme}

Now, we have all the necessary material to pass to the limit in the numerical scheme.

In the sequel,  let $(\Tau_m)_{m\in\mathbb{N}}$ be a sequence of admissible meshes of $\Lambda$ in the sense of Definition \ref{defmesh} such that the mesh size $h_m$ tends to $0$ when $m$ tends to $+\infty$ and let $(N_m)_{m\in\mathbb{N}} \subset\mathbb{N}^\star$ be a sequence with $\lim_{m\rightarrow+\infty} N_m=+\infty$, and $\Delta t_m:=\frac{T}{N_m}$.\\ 
For the sake of simplicity, for $m\in\mathbb{N}$, we shall use the notations $\Tau=\Tau_m$, $h=\operatorname{size}(\Tau_m)$, $\Delta t=\Delta t_m$, and $N=N_m$ when the $m$-dependency is not useful for the understanding of the reader.

\subsection{Weak convergence of finite-volume approximations}
Firstly, thanks to the bounds on the discrete solutions, we obtain the following weak convergences.

\begin{prop}\label{addreg u}
There exist not relabeled subsequences of $(u_{h,N}^r)_m$ and of $(u_{h,N}^l)_m$ respectively and a process $u\in L^2_{\mathcal{P}_T}\big(\Omega\times(0,T);H^1(\Lambda)\big)$ such that 
\[u_{h,N}^l\rightarrow u \ \text{and} \ \uhnr\rightarrow u,\text{ both weakly in }L^2(\Omega;L^2(0,T;L^2(\Lambda)))\text{ as }m\rightarrow+\infty.\]
\end{prop}

\begin{proof}
From Lemma \ref{210611_lem01} it follows that the sequences $(u_{h,N}^r)_m$, $(\uhnl)_m$ respectively are bounded in $L^2(\Omega;L^2(0,T;L^2(\Lambda)))$, thus, up to a not relabeled subsequence, they are weakly convergent in $L^2(\Omega;L^2(0,T;L^2(\Lambda)))$ towards possibly distinct elements $u$, $\tilde{u}$ respectively. Moreover, from Lemma \ref{remarkuhnrboundintomega} and the fact that (see \cite[Remark~2.7]{BNSZ22})
\begin{align}\label{linkgradsnh1}
\|\nabla^h u^r_{h,N}\|_{(L^2(\Lambda))^2}^2
=\sum_{\sigma=K|L\in\edgesint}\frac{d_{K|L}\times m_\sigma}{d}\left|d\frac{u^{n+1}_K-u^{n+1}_L}{d_{K|L}}\right|^2=d|u^r_{h,N}|_{1,h}^2,
\end{align}
it follows that
\[\| \nabla^h \uhnr\|^2_{L^2(\Omega; L^2(0,T;L^2(\Lambda)))}\leq d \cter{K1}.\]  
Consequently, there exists $\chi\in L^2(\Omega;L^2(0,T;L^2(\Lambda)))$ such that, passing to a not relabeled subsequence if necessary,
$\nabla^h \uhnr\rightarrow \chi$ weakly in $L^2(\Omega;L^2(0,T;L^2(\Lambda)))$ for $m\rightarrow+\infty$. With similar arguments as in \cite[Lemma 2]{eymardgallouet} and \cite[Theorem 14.3]{gal} we get the additional regularity $u\in L^2(\Omega;L^2( 0,T;H^1(\Lambda)))$ and $\chi=\nabla u$. Since, by Proposition~\ref{bounds},
\begin{align}\label{210824_05}
\erww{\|\uhnr-\uhnl\|_{L^2(0,T;\lzlambda)}^2}=\Delta t\erww{\sum_{n=0}^{N-1}\|u_h^{n+1}-u_h^n\|_\lzlambda^2}
\leq K_0\Delta t,
\end{align}
it follows that $(\uhnr-\uhnl)_m$ converges strongly to $0$ in $L^2(\Omega;L^2(0,T;L^2(\Lambda)))$ for $m\rightarrow+\infty$, hence also weakly and therefore $u=\tilde{u}$. Note that the predictability property of $u$ with values in $L^2(\Lambda)$ is inherited from $(u_{h,N}^l)_m$ at the limit. Now, since we also have $u\in L^2(\Omega;L^2( 0,T;H^1(\Lambda)))$, as pointed out in Remark \ref{240202_01}, the predictability of $u$ with values in $H^1(\Lambda)$ follows from the standard theory of Bochner spaces. 
\end{proof}

\begin{lem}\label{CVguhnl} There exist not relabeled subsequences of $(g(u_{h,N}^r))_m$ and $(g(u_{h,N}^l))_m$, and a  process $g_u$ in $L^2_{\mathcal{P}_T}\big(\Omega\times(0,T); H^1(\Lambda)\big)$ such that
$$g(u_{h,N}^r)\rightarrow g_u \text{ and } g(u_{h,N}^l)\rightarrow g_u,\text{ both weakly in }L^2(\Omega;L^2(0,T;L^2(\Lambda)))\text{ as }m\rightarrow+\infty.$$
\end{lem}

\begin{proof} The existence of a common weak limit $g_u$ in $L^2(\Omega;L^2( 0,T;L^2(\Lambda)))$ is a direct consequence of the boundedness results on $(g(u_{h,N}^r))_m$ and $(g(u_{h,N}^l))_m$ in the same space stated in Lemma \ref{boundguhnlr}, the Lipschitz property of $g$ and Inequality \eqref{210824_05}. Moreover, from Lemma \ref{boundsnguhnlr} and the equality (\ref{linkgradsnh1}) applied to $(g(u^r_{h,N}))_{h,N}$, it follows that 
\[\| \nabla^h g(\uhnr)\|^2_{L^2(\Omega; L^2(0,T;L^2(\Lambda)))}\leq d\cter{K3}.\]  Consequently, there exists $\tilde{\chi}\in L^2(\Omega;L^2(0,T;L^2(\Lambda)))$ such that, passing to a not relabeled subsequence if necessary,
$\nabla^h g(\uhnr)\rightarrow \tilde{\chi}$ weakly in $L^2(\Omega;L^2(0,T;L^2(\Lambda)))$ as $m\rightarrow+\infty$. With similar arguments as in \cite[Lemma 2]{eymardgallouet} and \cite[Theorem 14.3]{gal} we get the additional regularity $g_u\in L^2(\Omega;L^2( 0,T;H^1(\Lambda)))$ and $\tilde{\chi}=\nabla g_u$. Note that the predictability property of $g_u$ with values in $H^1(\Lambda)$ is obtained using  the same arguments as in the proof of Proposition \ref{addreg u}. 
\end{proof}
\begin{remark}
Note that the information $g_u\in L^2_{\mathcal{P}_T}\big(\Omega\times(0,T); H^1(\Lambda)\big)$ will be one of the keys points in the proof of Proposition \ref{PTTL} below.
\end{remark}

\begin{lem}\label{CVbetauhnl} There exist not relabeled subsequences of $(\beta(u_{h,N}^r))_m$ and $(\beta(u_{h,N}^l))_m$, and a process $\beta_u$ in $L^2_{\mathcal{P}_T}\big(\Omega\times(0,T); L^2(\Lambda)\big)$ such that
$$\beta(u_{h,N}^r)\rightarrow \beta_u \text{ and } \beta(u_{h,N}^l)\rightarrow \beta_u,\text{ both weakly in }L^2(\Omega;L^2(0,T;L^2(\Lambda)))\text{ as }m\rightarrow+\infty.$$
\end{lem}
\begin{proof} The weak convergence in $L^2(\Omega;L^2( 0,T;L^2(\Lambda)))$ of the sequences $(\beta(u_{h,N}^r))_m$ and  $(\beta(u_{h,N}^l))_m$ is a direct consequence of their boundedness properties in such a space stated in Lemma \ref{boundguhnlr}. The fact that the weak limit $\beta_u$ is common is due to Inequality \eqref{210824_05} and again the Lipschitz property of $\beta$. At last, the predictability property of $\beta_u$ with values in $L^2(\Lambda)$ is inherited from $(g(u_{h,N}^l))_m$ at the limit.
\end{proof}

\begin{prop}\label{PTTL} The weak limit $u$ of our finite-volume scheme \eqref{eq:def_u0}-\eqref{equationapprox} introduced  in Proposition \ref{addreg u} has $\mathds{P}$-a.s. continuous paths with values in $L^2(\Lambda)$ and satisfies for all $t\in [0,T]$,
\[u(t)=u_0+\int_0^t\Delta u(s)\,ds-\int_0^t\diver\big(\mathbf{v}(s,\cdot)u(s)\big)ds+\int_0^t g_u(s)\,dW(s)+\int_0^t \beta_u(s)\,ds,
\quad\]
in $L^2(\Lambda)$ and $\mathds{P}$-a.s. in $\Omega$, where $\Delta$ denotes the Laplace operator on $H^1(\Lambda)$ associated with the formal Neumann boundary conditions, and $g_u$ and $\beta_u$ respectively are given by Lemmas \ref{CVguhnl} and \ref{CVbetauhnl}.
\end{prop}

\begin{proof} Let $A\in\mathcal{A}$, $\xi\in \mathscr{D}(\re)$ with $\xi(T)=0$ and $\varphi\in \mathscr{D}(\re^d)$ with $\nabla\varphi\cdot\mathbf{n}=0$ on $\partial\Lambda$, where we denote $\mathscr{D}(D):=\mathscr{C}_c^\infty(D)$ for any open subset $D\subseteq \re^m, m\in\na^{\ast}$. We introduce the discrete function $\varphi_h : \Lambda\rightarrow \R$ defined by $\displaystyle\varphi_h(x)=\sum_{K\in \Tau}\mathds{1}_K(x)\varphi(x_K)$ for any $x\in \Lambda$. \\
For $K\in\Tau$, $n\in\{0,\dots,N-1\}$ and $t\in[t_n,t_{n+1})$ we multiply \eqref{equationapproxbis} with $\mathds{1}_A\xi(t)\varphi(x_K)$ to obtain
\begin{align}\label{220814}
\begin{split}
&\mathds{1}_A\xi(t)\frac{m_K}{\Delta t}\Big(u_K^{n+1}-u_K^n-g(u_K^n)\big(W^{n+1}-W^n\big)\Big)\varphi(x_K)\\
&+\mathds{1}_A\xi(t)\sum_{\sigma=K|L\in\edgesint\cap\edges_K}\frac{m_\sigma}{\dkl}(u_K^{n+1}-u_L^{n+1})\varphi(x_K)\\
&+\mathds{1}_A\xi(t)\sum_{\sigma\in\edgesint\cap\edges_K} m_\sigma \vksnp\big(u^{n+1}_{\sigma}-u^{n+1}_K\big)\varphi(x_K)\\
=\,&\mathds{1}_A\xi(t)m_K\beta(\unpk)\varphi(x_K).
\end{split}
\end{align}
Firstly, we sum \eqref{220814} over each control volume $K\in\Tau$, we integrate over each time interval $[t_{n},t_{n+1}]$, then we sum over $n=0,\dots,N-1$, and finally we take the expectation to obtain 
\begin{equation}\label{discretequforlimit}
S_{1,m}+S_{2,m}+S_{3,m}+S_{4,m}=S_{5,m}
\end{equation}
where
\begin{eqnarray*}
S_{1,m}&=&\mathbb{E}\left[\sum_{n=0}^{N-1}\int_{t_n}^{t_{n+1}}\sum_{K\in\Tau}\mathds{1}_A\xi(t)m_K\frac{u_K^{n+1}-u_K^n}{\Delta t}\varphi(x_K)\,dt\right]\\
S_{2,m}&=&-\mathbb{E}\left[\sum_{n=0}^{N-1}\int_{t_n}^{t_{n+1}}\sum_{K\in\Tau}\mathds{1}_A\xi(t)m_Kg(u_K^n)\frac{W^{n+1}-W^n}{\Delta t}\varphi(x_K)\,dt\right]\\
S_{3,m}&=&\mathbb{E}\left[\sum_{n=0}^{N-1}\int_{t_n}^{t_{n+1}}\sum_{K\in\Tau}\mathds{1}_A\xi(t)\sum_{\sigma=K|L\in\edgesint\cap\edges_K}\frac{m_\sigma}{\dkl}(u_K^{n+1}-u_L^{n+1})\varphi(x_K)\,dt\right]\\
S_{4,m}&=&\mathbb{E}\left[\sum_{n=0}^{N-1}\int_{t_n}^{t_{n+1}}\sum_{K\in\Tau}\mathds{1}_A\xi(t)\sum_{\sigma\in\edgesint\cap\edges_K} m_\sigma \vksnp\big(u^{n+1}_{\sigma}-u^{n+1}_K\big)\varphi(x_K)\, dt\right]\\
S_{5,m}&=&\mathbb{E}\left[\sum_{n=0}^{N-1}\int_{t_n}^{t_{n+1}}\sum_{K\in\Tau}\mathds{1}_A\xi(t)m_K \beta(\unpk)\varphi(x_K)\,dt\right].
\end{eqnarray*}
Let us study separately the limit as $m$ goes to $+\infty$ of $S_{1,m}$, $S_{2,m}$, $S_{3,m}$, $S_{4,m}$ and $S_{5,m}$.\medskip\\
$\bullet$ Study of $S_{1,m}$: By using Proposition \ref{addreg u} and a discrete integration by parts formula, one shows that up to a subsequence denoted in the same way 
\begin{eqnarray*}
S_{1,m}\xrightarrow[m\rightarrow +\infty]{}-\mathbb{E}\left[\mathds{1}_A\int_0^T\int_\Lambda u(t,x)\xi'(t)\varphi(x)\,dx\,dt\right]-\mathbb{E}\left[\mathds{1}_A\int_\Lambda u_0(x)\xi(0)\varphi(x)\,dx\right].
\end{eqnarray*}
Indeed, if we consider the following discrete integration by parts formula
\begin{eqnarray*}
\sum_{n=1}^{N}a_n(b_n-b_{n-1})=a_Nb_N-a_0b_0-\sum_{n=0}^{N-1}b_n(a_{n+1}-a_n)
\end{eqnarray*}
applied to $a_n=\unk$ and $b_n=\xi(\tn)$, one has since $\xi(T)=0$
\begin{eqnarray*}
&&\E\left[\mathds{1}_A\int_{\Delta t}^T\int_{\Lambda}\uhnl(t,x)\xi'(t-\Delta t)\varphi(x)\,dx\,dt\right]\\
&=&\E\left[\mathds{1}_A\sum_{n=1}^{N-1}\sum_{K\in\Tau}\int_{\tn}^{\tnp}\int_K\unk\xi'(t-\Delta t)\varphi(x)\,dx\,dt\right]\\
&=&\E\left[\mathds{1}_A\sum_{n=1}^{N-1}\sum_{K\in\Tau}\unk(\xi(\tn)-\xi(t_{n-1}))\int_K\varphi(x)\,dx\right]\\
&=&-\E\left[\mathds{1}_A\sum_{K\in\Tau} u^0_K\xi(0)\int_K\varphi(x)\,dx\right]-\E\left[\mathds{1}_A\sum_{K\in\Tau}\sum_{n=0}^{N-1}(\unpk-\unk)\xi(\tn)\int_K\varphi(x)\,dx\right].
\end{eqnarray*}
Firstly, thanks to the weak convergence of $(\uhnl)_m$ towards $u$ in $L^2(\Omega;L^2(0,T;L^2(\Lambda)))$ given by Proposition \ref{addreg u}, one gets
$$\E\left[\mathds{1}_A\int_{\Delta t}^T\int_{\Lambda}\uhnl(t,x)\xi'(t-\Delta t)\varphi(x)\,dx\,dt\right]\xrightarrow[m\rightarrow +\infty]{}\mathbb{E}\left[\mathds{1}_A\int_0^T\int_\Lambda u(t,x)\xi'(t)\varphi(x)\,dx\,dt\right].$$
Secondly, we have
\begin{eqnarray*}
&&\left|\E\left[\mathds{1}_A\sum_{K\in\Tau} u^0_K\xi(0)\int_K\varphi(x)\,dx\right]-\mathbb{E}\left[\mathds{1}_A\int_\Lambda u_0(x)\xi(0)\varphi(x)\,dx\right]\right|\\
&=&\left|\E\left[\mathds{1}_A\sum_{K\in\Tau} \xi(0)\int_K(u^0_K-u_0(x))\varphi(x)\,dx\right]\right|\\
&=&\left|\E\left[\mathds{1}_A\sum_{K\in\Tau} \xi(0)\int_K\left(\frac{1}{m_K}\int_Ku_0(y)-u_0(x)\,dy\right) \varphi(x)\,dx\right]\right|\\
&\leq&||\xi||_\infty||\varphi||_\infty\E\left[\sum_{K\in\Tau} \int_K\left(\frac{1}{m_K}\int_K|u_0(y)-u_0(x)|\,dy\right)\,dx\right],\\
\end{eqnarray*}
and since by Assumption $\mathscr{A}_1$, $u_0$ belongs particularly to $L^1(\Omega; L^1(\Lambda))$, this last term tends to $0$ as $m$ goes to $+\infty$ thank's to Lebesgue's dominated \smallskip convergence theorem.\smallskip \\
$\bullet$ Study of $S_{2,m}$: Thanks to Lemma \ref{CVguhnl} and the properties of the stochastic integral, one shows that, up to a subsequence denoted in the same way,
\begin{eqnarray*}
S_{2,m}\xrightarrow[m\rightarrow +\infty]{}\mathbb{E}\left[\mathds{1}_A\int_0^T\int_\Lambda \int_0^t g_u(s,x)\,dW(s)\xi'(t)\varphi(x)\,dx\,dt\right].
\end{eqnarray*}
To do so, we introduce the following terms
\begin{eqnarray*}
\tilde{S}_{2,m}&=&-\mathbb{E}\left[\sum_{n=0}^{N-1}\sum_{K\in\Tau}m_K\mathds{1}_A\xi(\tn)g(u_K^n)\big(W^{n+1}-W^n\big)\varphi(x_K)\right]\\
\hat{S}_{2,m}&=&-\mathbb{E}\left[\mathds{1}_A\int_{\Lambda}\int_0^T\xi(t)g(\uhnl(t,x))\,dW(t)\varphi(x)\,dx\right]\\
\end{eqnarray*}
and decompose $S_{2,m}$ as $$S_{2,m}=(S_{2,m}-\tilde{S}_{2,m})+(\tilde{S}_{2,m}-\hat{S}_{2,m})+\hat{S}_{2,m}$$
and study separately each term of this decomposition. \\
Firstly, we have by applying successively Cauchy-Schwarz inequality on $\Omega$ and then on the sum over $n\in\{0,..,N-1\}$ and $K\in \Tau$:
\begin{eqnarray*}
&&|S_{2,m}-\tilde{S}_{2,m}|^2\\
&=&\left|\mathbb{E}\left[\sum_{n=0}^{N-1}\int_{t_n}^{t_{n+1}}\sum_{K\in\Tau}m_K\mathds{1}_A\big(\xi(t)-\xi(\tn)\big)g(u_K^n)\frac{W^{n+1}-W^n}{\Delta t}\varphi(x_K)\,dt\right]\right|^2\\
&\leq&\left(\mathbb{E}\left[\sum_{n=0}^{N-1}\int_{t_n}^{t_{n+1}}\sum_{K\in\Tau}m_K\left|\mathds{1}_A\big(\xi(t)-\xi(\tn)\big)g(u_K^n)\frac{W^{n+1}-W^n}{\Delta t}\varphi(x_K)\right|\,dt\right]\right)^2\\
&\leq&\frac{(\Delta t)^2}{(\Delta t)^2}||\xi'||^2_\infty||\varphi||^2_\infty\left(\sum_{n=0}^{N-1}\dlt\sum_{K\in\Tau}m_K\mathbb{E}\Big[\left|\mathds{1}_Ag(u_K^n)\big(W^{n+1}-W^n\big)\right|\Big]\right)^2\\
&\leq&||\xi'||^2_\infty||\varphi||^2_\infty\left(\sum_{n=0}^{N-1}\sum_{K\in\Tau}\dlt m_K\left(\mathbb{E}[\mathds{1}_A]\right)^{\frac{1}{2}}\times  \left\{\mathbb{E}\Big[g^2(u_K^n)\big(W^{n+1}-W^n\big)^2\Big]\right\}^{\frac{1}{2}}\right)^2\\
&\leq&||\xi'||^2_\infty||\varphi||^2_\infty \left(\sum_{n=0}^{N-1}\sum_{K\in\Tau}\dlt m_K \right)\times\left(\sum_{n=0}^{N-1}\sum_{K\in\Tau}\dlt m_K \mathbb{E}\Big[g^2(u_K^n)\big(W^{n+1}-W^n\big)^2\Big] \right) \\
&\leq&||\xi'||^2_\infty||\varphi||^2_\infty T|\Lambda|\sum_{n=0}^{N-1}\sum_{K\in\Tau}\dlt m_K \mathbb{E}\big[g^2(u_K^n)\big]\times \mathbb{E}\Big[\big(W^{n+1}-W^n\big)^2\Big]  \\
&\leq&\Delta t ||\xi'||^2_\infty  ||\varphi||_\infty^2 |\Lambda| TC_{L_g}\big(1+ ||\uhnl||^2_{L^2(\Omega;L^2(0,T;\lzlambda))}\big) ,
\end{eqnarray*}
which tends to $0$ as $m$ goes to $+\infty$. \\
Secondly, we start by rewriting $\tilde{S}_{2,m}-\hat{S}_{2,m}$ in the following manner
\begin{align}\label{decomps2m}
\begin{split}
-(\tilde{S}_{2,m}-\hat{S}_{2,m})&=\mathbb{E}\left[\mathds{1}_A\sum_{n=0}^{N-1}\sum_{K\in\Tau}m_K\int _{\tn}^{\tnp}\xi(\tn)g(u_K^n)\,dW(t)\varphi(x_K)\right]\\
&\quad-\mathbb{E}\left[\mathds{1}_A\sum_{n=0}^{N-1}\sum_{K\in\Tau}m_K\int _{\tn}^{\tnp}\xi(t)g(\unk)\,dW(t)\varphi(x_K)\right]\\
&\quad+\mathbb{E}\left[\mathds{1}_A\sum_{n=0}^{N-1}\sum_{K\in\Tau}m_K\int _{\tn}^{\tnp}\xi(t)g(\unk)\,dW(t)\varphi(x_K)\right]\\
&\quad-\mathbb{E}\left[\mathds{1}_A\int_\Lambda\int _{0}^{T}\xi(t)g(\uhnl(t,x))\,dW(t)\varphi(x)\,dx\right].
\end{split}
\end{align}
Our aim is to analyze the difference between the two first terms (the two last terms, respectively) on the right hand side of \eqref{decomps2m}. To do so, we involve the discrete function $\varphi_h$ introduced at the beginning of the proof. Using successively Cauchy-Schwarz inequality on $\Omega\times \Lambda$ and It\^o isometry, one gets since $\mathds{P}(A)\leq 1$
\begin{eqnarray*}
&&\left|\mathbb{E}\left[\mathds{1}_A\sum_{n=0}^{N-1}\sum_{K\in\Tau}m_K\int _{\tn}^{\tnp}\big(\xi(\tn)-\xi(t)\big)g(u_K^n)\,dW(t)\varphi(x_K)\right]\right|\\
&\leq&\sum_{n=0}^{N-1}\left|\mathbb{E}\left[\mathds{1}_A\int_\Lambda\int _{\tn}^{\tnp}\big(\xi(\tn)-\xi(t)\big)g(\uhnl(t,x))\varphi_h(x)\,dW(t)\,dx\right]\right|\\
&\leq&\sqrt{|\Lambda|}\sum_{n=0}^{N-1}\left(\mathbb{E}\left[\int_\Lambda\left\{\int _{\tn}^{\tnp}\big(\xi(\tn)-\xi(t)\big)g(\uhnl(t,x))\varphi_h(x)\,dW(t)\right\}^2\,dx\right]\right)^{\frac{1}{2}}\\
&=&\sqrt{|\Lambda|}\sum_{n=0}^{N-1}\left(\mathbb{E}\left[\int_\Lambda\int _{\tn}^{\tnp}\Big\{\big(\xi(\tn)-\xi(t)\big)g(\uhnl(t,x))\varphi_h(x)\Big\}^2\,dt\,dx\right]\right)^{\frac{1}{2}}\\
&\leq& \sqrt{\Delta t}\sqrt{|\Lambda|}||\xi'||_\infty||\varphi||_\infty L_g\sum_{n=0}^{N-1}\Delta t\left( \sup_{t\in [0,T]}\mathbb{E}\left[\int_\Lambda |\uhnl(t,x)|^2\,dx\right]\right)^{\frac{1}{2}}
\end{eqnarray*}
which tends to $0$ thanks to the bound given by Proposition \ref{bounds}. Let us now study the difference between the two last terms of the right hand side of \eqref{decomps2m}. Using again Cauchy-Schwarz inequality on $\Omega\times \Lambda$ and It\^o isometry, one arrives at
\begin{eqnarray*}
&&\left|\mathbb{E}\left[\mathds{1}_A\int_\Lambda\int _{0}^{T}\xi(t)g(\uhnl(t,x))\big(\varphi_h(x)-\varphi(x)\big) \,dW(t)\,dx\right]\right|^2\\
&\leq& |\Lambda|\times \mathbb{E}\left[\int_\Lambda\left\{\int _{0}^{T}\xi(t)g(\uhnl(t,x))\big(\varphi_h(x)-\varphi(x)\big)\,dW(t)\right\}^2\,dx\right]\\
&=& |\Lambda|\times \mathbb{E}\left[\int_\Lambda\int _{0}^{T}\left\{\xi(t)g(\uhnl(t,x))\big(\varphi_h(x)-\varphi(x)\big)\right\}^2\,dt\,dx\right]\\
&\leq& h^2 ||\xi||^2_\infty ||\nabla \varphi||^2_\infty |\Lambda| \mathbb{E}\left[\int_\Lambda\int _{0}^{T}C_{L_g}(1+|\uhnl(t,x)|^2)\,dt\,dx\right],\\
\end{eqnarray*}
which tends to $0$ as $m$ goes to $+\infty$ thanks to the control on $ \mathbb{E}\left[\int_\Lambda\int _{0}^{T}|\uhnl(t,x)|^2\,dt\,dx\right]$ given by Lemma \ref{210611_lem01}. Finally, one can affirm that $\tilde{S}_{2,m}-\hat{S}_{2,m}\xrightarrow[m\rightarrow +\infty]{} 0$.\\
Thirdly, we show that the following convergence result holds: 
$$\hat{S}_{2,m}\xrightarrow[m\rightarrow +\infty]{}-\mathbb{E}\left[\mathds{1}_A\int_{\Lambda}\int_0^T\xi(s)g_u(s,x)\,dW(s)\varphi(x)\,dx\right].$$
To do so, recall that, thanks to It\^{o} isometry, the following linear application is continuous:
\begin{eqnarray*}
I_T:L^2_{\mathcal{P}_T}\big(\Omega\times(0,T); L^2(\Lambda)\big)&\rightarrow& L^2(\Omega; L^2(\Lambda))\\
X&\mapsto& \int_0^T X(\omega,t,x)\,dW(t).
\end{eqnarray*}
Set $(X_m)_m=(g(\uhnl)\xi)_m$, then thanks to Lemma \ref{CVguhnl}, up to a subsequence denoted in the same way, $(X_m)_m$ converges weakly towards $g_u\xi$ in $L^2_{\mathcal{P}_T}\big(\Omega\times(0,T); L^2(\Lambda)\big)$ and so since $\varphi\mathds{1}_A \in L^2(\Omega; L^2(\Lambda))$ one gets
\begin{eqnarray*}
\hat{S}_{2,m}&=&-\int_{\Omega}\int_{\Lambda}I_T(X_m)(\omega,x)\varphi(x)\mathds{1}_A(\omega)\,dx\,d\mathds{P}(\omega)\\
&\xrightarrow[m\rightarrow +\infty]{}&-\int_{\Omega}\int_{\Lambda}I_T(g_u\xi)(\omega,x)\varphi(x)\mathds{1}_A(\omega)\,dx\,d\mathds{P}(\omega)\\
&=&-\mathbb{E}\left[\mathds{1}_A\int_{\Lambda}\int_0^T\xi(s)g_u(s,x)\,dW(s)\varphi(x)\,dx\right]
\end{eqnarray*}
Finally, thanks to It\^{o} formula (see \cite{DPZ14} Theorem 4.17 p.105) we may apply a rule of stochastic integration by parts to conclude
$$-\mathbb{E}\left[\mathds{1}_A\hspace*{-0.1cm}\int_{\Lambda}\int_0^T\hspace*{-0.1cm}\xi(s)g_u(s,x)\,dW(s)\varphi(x)\,dx\right]\hspace*{-0.1cm}=\hspace*{-0.1cm}\mathbb{E}\left[\mathds{1}_A\hspace*{-0.1cm}\int_{\Lambda}\int_0^T\int_0^t\hspace*{-0.1cm}g_u(s,x)\,dW(s)\xi'(t)\varphi(x)\,dt\,dx\right].$$
$\bullet$ Study of $S_{3,m}$: One shows that 
\begin{align*}
S_{3,m}\xrightarrow[m\rightarrow +\infty]{} -\mathbb{E}\left[\mathds{1}_A\int_0^T\int_\Lambda\xi(t)\Delta\varphi(x)u(t,x)\,dx\,dt\right].
\end{align*}
Indeed, thanks to a discrete integration by part formula, $S_{3,m}$ can be written as
\[
 S_{3,m}
 =\mathbb{E}\left[\sum_{n=0}^{N-1}\int_{t_n}^{t_{n+1}}\mathds{1}_A\xi(t)
 \sum_{K\in\Tau}u_K^{n+1}
 \sum_{\sigma=K|L\in\edgesint\cap\edges_K}m_\sigma\left(\frac{\varphi(x_K)-\varphi(x_L)}{\dkl}\right)\,dt\right].
\]
Then, since $\nabla\varphi\cdot\mathbf{n}=0$ on $\partial\Lambda$, thanks to the Green-Ostrogradski Theorem one has,
\[
 \int_K\Delta\varphi(x)\,dx
 = \int_{\partial K}\nabla\varphi(x)\cdot\mathbf{n}(x)\,d\gamma(x)
 = \sum_{\sigma\in\edgesint\cap\edges_K} \int_\sigma \nabla\varphi(x)\cdot\mathbf{n}_{K,\sigma}\,d\gamma(x).
\]
Thus, we have
\begin{align*}
 S_{3,m}&= \mathbb{E}\left[\sum_{n=0}^{N-1}\int_{t_n}^{t_{n+1}}\mathds{1}_A\xi(t)
 \sum_{K\in\Tau} u_K^{n+1}\sum_{\sigma=K|L\in\edgesint\cap\edges_K}m_\sigma\left(\frac{\varphi(x_K)-\varphi(x_L)}{\dkl}\right)\,dt\right]\\
 &-\mathbb{E}\left[\sum_{n=0}^{N-1}\int_{t_n}^{t_{n+1}}\mathds{1}_A\xi(t)
 \sum_{K\in\Tau} u_K^{n+1}
 \left(\int_K\Delta\varphi(x)\,dx - \hspace*{-0.4cm}\sum_{\sigma\in\edgesint\cap\edges_K} \int_\sigma \nabla\varphi(x)\cdot\mathbf{n}_{K,\sigma}\,d\gamma(x) \right)\,dt\right] \\
 &= -\mathbb{E}\left[\int_0^T\mathds{1}_A\xi(t)\int_\Lambda \uhnr(t,x)\Delta\varphi(x)\,dx\,dt\right]\\
&+ \mathbb{E}\left[\sum_{n=0}^{N-1}\int_{t_n}^{t_{n+1}}\mathds{1}_A\xi(t)
\sum_{\sigma=K|L\in\edgesint} m_\sigma (u_K^{n+1}-u_L^{n+1}) R_\sigma^\varphi
 \,dt\right],
\end{align*}
with
\[
 R_\sigma^\varphi
 = \frac1{m_\sigma} \int_\sigma \nabla\varphi(x)\cdot\mathbf{n}_{K,\sigma}\,d\gamma(x)
 - \frac{\varphi(x_L)-\varphi(x_K)}{\dkl}.
\]
Using Proposition \ref{addreg u} and up to a subsequence denoted in the same way, one gets
\begin{align*}
-\mathbb{E}\left[\int_0^T\int_\Lambda\mathds{1}_A\xi(t)u_{h,N}^r(t,x)\Delta\varphi(x)\,dx\,dt\right]
\xrightarrow[m\rightarrow +\infty]{}-\mathbb{E}\left[\int_0^T\int_\Lambda\mathds{1}_A\xi(t)u(t,x)\Delta\varphi(x)\,dx\,dt\right].
\end{align*}
Note that one is able to control the deterministic rest $R_\sigma^\varphi$ since for any $\sigma=K|L\in\edgesint$, the orthogonality condition implies $x_L-x_K=\dkl\mathbf{n}_{K,\sigma}$, thus thanks to the Taylor formula there exists a constant $C_{\varphi}>0$ only depending on $\varphi$ such that for any $\sigma\in \edgesint$ one has
\[
 |R_\sigma^\varphi| \leq C_\varphi h.
\]
Therefore, using the fact that $\di\sum_{\sigma=K|L\in\edgesint}\frac{m_\sigma\dkl}{d}\leq |\Lambda|$, one finally obtains thanks to the Cauchy-Schwarz inequality and Lemma \ref{remarkuhnrboundintomega} \begin{align*}
 &\left|\mathbb{E}\left[\sum_{n=0}^{N-1}\int_{t_n}^{t_{n+1}}\mathds{1}_A\xi(t)
\sum_{\sigma=K|L\in\edgesint} m_\sigma (u_K^{n+1}-u_L^{n+1}) R_\sigma^\varphi
 \,dt\right]\right|\\
 &\leq C_\varphi h\mathbb{E}\left[\sum_{n=0}^{N-1}\int_{t_n}^{t_{n+1}}\mathds{1}_A |\xi(t)|
 \left(\sum_{\sigma=K|L\in\edgesint}m_\sigma\dkl\right)^\frac12 \left(\sum_{\sigma=K|L\in\edgesint}m_{\sigma}\frac{|u_K^{n+1}-u_L^{n+1}|^2}{\dkl}\right)^\frac12\,dt\right]\\
 &\leq \sqrt d C_\varphi |\Lambda|^\frac12 h\mathbb{E}\left[\int_0^T\mathds{1}_A |\xi(t)||\uhnr(t)|_{1,h}\,dt\right]\\
 &\leq h\sqrt d C_\varphi |\Lambda|^\frac12 \|\xi\mathds{1}_A\|_{L^2(\Omega\times(0,T))}\left(\mathbb{E}\left[\int_0^T|\uhnr(t)|^2_{1,h}\,dt\right]\right)^{\frac{1}{2}}
\xrightarrow[m\rightarrow +\infty]{} 0.
\end{align*}
$\bullet$ Study of $S_{4,m}$: Adapting to the stochastic case arguments exposed in \cite{gal} p.774-776 in the deterministic and elliptic case, one shows that 
$$S_{4,m}\xrightarrow[m\rightarrow +\infty]{} -\mathbb{E}\left[\mathds{1}_A\int_0^T\int_{\Lambda}\big(\ve(t,x)u(t,x)\big)\cdot\nabla \varphi(x)\xi(t)\,dx\,dt\right].$$
To do so, we decompose the term
$$S_{4,m}=\mathbb{E}\left[\sum_{n=0}^{N-1}\int_{t_n}^{t_{n+1}}\mathds{1}_A\xi(t)\sum_{K\in\Tau}\sum_{\sigma\in\edgesint\cap\edges_K} m_\sigma \vksnp\big(u^{n+1}_{\sigma}-u^{n+1}_K\big)\varphi(x_K)\,dt\right]$$
in the following manner: $S_{4,m}=S_{4,m}-\tilde{S}_{4,m}+\tilde{S}_{4,m}-\hat{S}_{4,m}+\hat{S}_{4,m}$
where 
\begin{align*}
\tilde{S}_{4,m}&=\mathbb{E}\left[\sum_{n=0}^{N-1}\int_{t_n}^{t_{n+1}}\hspace*{-0.25cm}\mathds{1}_A\xi(t)\hspace*{-0.1cm}\sum_{K\in\Tau}\hspace*{-0.1cm}\sum_{\sigma\in\edgesint\cap\edges_K} \hspace*{-0.4cm}\big(u^{n+1}_{\sigma}-u^{n+1}_K\big)
\frac{1}{\dlt}\int_{\tn}^{\tnp}\hspace*{-0.2cm} \int_\sigma\ve(s,x)\cdot\nks \varphi(x)\,d\gamma(x)dsdt\right]\\
\hat{S}_{4,m}&=-\mathbb{E}\left[\mathds{1}_A\int_0^T \int_{\Lambda}\xi(t) \uhnr(t,x)\diver\big(\ve(t,x) \varphi(x)\big)\,dx\,dt\right].
\end{align*}
Firstly, one shows that $|S_{4,m}-\tilde{S}_{4,m}|$ tends to $0$ as $m$ goes to $+\infty$. Indeed, by noticing that 
\begin{eqnarray*}
\left|\frac{1}{\ms\dlt}\int_{\tn}^{\tnp} \int_\sigma\ve(s,x)\cdot \nks \big(\varphi(x_K)-\varphi(x)\big)\,d\gamma(x)ds\right|&\leq &h ||\ve||_\infty ||\nabla \varphi||_\infty 
\end{eqnarray*}
one obtains 
\begin{eqnarray*}
|S_{4,m}-\tilde{S}_{4,m}|&=&\bigg|\mathbb{E}\Big[\sum_{n=0}^{N-1}\int_{t_n}^{t_{n+1}}\mathds{1}_A\xi(t)\sum_{K\in\Tau}\sum_{\sigma\in\edgesint\cap\edges_K} \ms\big(u^{n+1}_{\sigma}-u^{n+1}_K\big)\\
&&\hspace*{1cm}\times\frac{1}{\ms\dlt}\int_{\tn}^{\tnp}\int_\sigma\ve(s,x)\cdot \nks \big(\varphi(x_K)-\varphi(x)\big)\,d\gamma(x)\,ds\,dt\Big]\bigg|\\
&\leq&||\ve||_\infty ||\nabla \varphi||_\infty ||\xi||_\infty h \sum_{n=0}^{N-1}\sum_{K\in\Tau}\sum_{\sigma\in\edgesint\cap\edges_K} \dlt \ms \mathbb{E}\left[|\unps-\unpk|\right]\\
&\leq& ||\ve||_\infty ||\nabla \varphi||_\infty ||\xi||_\infty h \left(\sum_{n=0}^{N-1}\sum_{\sigma=K|L\in\edgesint} \dlt \ms \dkl\right)^{\frac{1}{2}}\\
&&\times \left(\sum_{n=0}^{N-1}\sum_{\sigma=K|L\in\edgesint} \dlt \frac{\ms}{\dkl} \mathbb{E}\Big[|\unpk-\unpl|^2\Big]\right)^{\frac{1}{2}} \\
&\leq&h  ||\ve||_\infty ||\nabla \varphi||_\infty ||\xi||_\infty \sqrt{d T|\Lambda|} \left(\int_0^T\erwb|\uhnr(t)|_{1,h}^2\erwe dt\right)^{\frac{1}{2}}\xrightarrow[m\rightarrow +\infty]{} 0.
\end{eqnarray*}
Secondly, one shows that $|\tilde{S}_{4,m}-\hat{S}_{4,m}|$ tends to $0$ as $m$ goes to $+\infty$. Indeed, for any $\sigma=K|L\in\edgesint $, since $\mathbf{n}_{K,\sigma}=-\mathbf{n}_{L,\sigma}$ then $\vksnp=-v_{L,\sigma}^{n+1}$ and so one remarks that $\tilde{S}_{4,m}$ can be rewritten as 
\begin{eqnarray*}
\tilde{S}_{4,m}=-\mathbb{E}\left[\sum_{n=0}^{N-1}\int_{t_n}^{t_{n+1}}\hspace*{-0.2cm}\mathds{1}_A\xi(t)\sum_{K\in\Tau}\sum_{\sigma\in\edgesint\cap\edges_K} \hspace*{-0.4cm}u^{n+1}_K
\frac{1}{\dlt}\int_{\tn}^{\tnp}\hspace*{-0.2cm} \int_\sigma\ve(s,x)\cdot \nks \varphi(x)\,d\gamma(x)\,ds\,dt\right].
\end{eqnarray*}
In this way, using the divergence-free property of $\ve$, one gets
\begin{eqnarray*}
\tilde{S}_{4,m}&=&-\mathbb{E}\left[\sum_{n=0}^{N-1}\int_{t_n}^{t_{n+1}}\hspace*{-0.2cm}\mathds{1}_A\xi(t)\sum_{K\in\Tau} u^{n+1}_K
\frac{1}{\dlt}\int_{\tn}^{\tnp} \int_{\partial K}\ve(s,x)\cdot \mathbf{ n}_K(x) \varphi(x)\,d\gamma(x)\,ds\,dt\right]\\
&=& -\mathbb{E}\left[\sum_{n=0}^{N-1}\int_{t_n}^{t_{n+1}}\hspace*{-0.2cm}\mathds{1}_A\xi(t)\sum_{K\in\Tau} u^{n+1}_K
\frac{1}{\dlt}\int_{\tn}^{\tnp} \int_{K}\diver\big(\ve(s,x) \varphi(x)\big)\,dx\,ds\,dt\right]\\
&=& -\mathbb{E}\left[\sum_{n=0}^{N-1}\int_{t_n}^{t_{n+1}}\hspace*{-0.2cm}\mathds{1}_A\xi(t)\sum_{K\in\Tau} u^{n+1}_K
\frac{1}{\dlt}\int_{\tn}^{\tnp} \int_{K}\ve(s,x)\cdot \nabla \varphi(x)\,dx\,ds\,dt\right].
\end{eqnarray*} 
In the same manner, we write $\hat{S}_{4,m}$ as
$$\hat{S}_{4,m}=-\mathbb{E}\left[\mathds{1}_A\int_0^T \int_{\Lambda}\xi(t) \uhnr(t,x)\ve(t,x)\cdot \nabla \varphi(x)\,dx\,dt\right]$$
and then 
\begin{eqnarray*}
&&\left|\tilde{S}_{4,m}-\hat{S}_{4,m}\right|\\
&=&
\left|\mathbb{E}\left[\mathds{1}_A\sum_{n=0}^{N-1}\sum_{K\in\Tau} u^{n+1}_K\hspace*{-0.1cm}\int_{t_n}^{t_{n+1}}\hspace*{-0.2cm}\xi(t)\hspace*{-0.1cm}\int_K \left\{\ve(t,x)\hspace*{-0.1cm}\cdot\hspace*{-0.1cm} \nabla \varphi(x)-
\frac{1}{\dlt}\int_{\tn}^{\tnp} \ve(s,x)\cdot\hspace*{-0.1cm}\nabla\hspace*{-0.01cm} \varphi(x)\,ds\hspace*{-0.1cm}\right\}\hspace*{-0.1cm}dxdt\right]\right|\\
&=&\left|\mathbb{E}\left[\mathds{1}_A\sum_{n=0}^{N-1}\sum_{K\in\Tau} u^{n+1}_K\int_{t_n}^{t_{n+1}}\hspace*{-0.2cm}\xi(t)\int_K 
\frac{1}{\dlt}\int_{\tn}^{\tnp} \Big\{\ve(t,x)-\ve(s,x)\Big\}\cdot\nabla \varphi(x)\,ds\,dx\,dt\right]\right|\\
&\leq&\dlt ||\xi||_\infty ||\partial_t \ve||_\infty ||\nabla \varphi ||_\infty  ||\uhnr||_{L^2(\Omega;L^2(0,T;L^2(\Lambda)))}\sqrt{T|\Lambda|} \xrightarrow[m\rightarrow +\infty]{} 0.
\end{eqnarray*} 
Thirdly, owing to the weak convergence of $(\uhnr)_m$ towards $u$ in $L^2(\Omega;L^2(0,T;L^2(\Lambda)))$ given by Proposition \ref{addreg u}, one gets directly that
\begin{eqnarray*}
\hat{S}_{4,m}&=&-\mathbb{E}\left[\mathds{1}_A\int_0^T \int_{\Lambda} \uhnr(t,x)\diver\big(\ve(t,x) \varphi(x)\big)\xi(t)\,dx\,dt\right]\\
&\xrightarrow[m\rightarrow +\infty]{}&-\mathbb{E}\left[\mathds{1}_A\int_0^T \int_{\Lambda} u(t,x)\diver\big(\ve(t,x) \varphi(x)\big)\xi(t)\,dx\,dt\right]\\
&=&-\mathbb{E}\left[\mathds{1}_A\int_0^T \int_{\Lambda}\big(\ve(t,x) u(t,x)\big) \cdot \nabla\varphi(x)\xi(t)\,dx\,dt\right]
\end{eqnarray*} 
$\bullet$ Study of $S_{5,m}$: Thanks to Lemma \ref{CVbetauhnl}, one shows that 
$$S_{5,m}\xrightarrow[m\rightarrow +\infty]{} \mathbb{E}\left[\mathds{1}_A\int_0^T \int_{\Lambda} \beta_u(t,x) \varphi(x)\xi(t)\,dx\,dt\right].$$
Indeed, let us remark that $S_{5,m}$ can be decomposed as $S_{5,m}=S_{5,m}-\tilde{S}_{5,m}+\tilde{S}_{5,m}$ where
$$\tilde{S}_{5,m}=\mathbb{E}\left[\mathds{1}_A\int_0^T\int_{\Lambda}\xi(t) \beta(\uhnr)\varphi(x)\,dx\,dt\right]\xrightarrow[m\rightarrow +\infty]{}\mathbb{E}\left[\mathds{1}_A\int_0^T \int_{\Lambda} \beta_u(t,x) \varphi(x)\xi(t)\,dx\,dt\right]$$
 and
\begin{eqnarray*}
|S_{5,m}-\tilde{S}_{5,m}|&=&\left|\mathbb{E}\left[\mathds{1}_A\sum_{n=0}^{N-1}\sum_{K\in\Tau}\int_{t_n}^{t_{n+1}}\int_{K}\xi(t) \beta(\unpk)\big(\varphi(x_K)-\varphi(x)\big)\,dx\,dt\right]\right|\\
&\leq& h||\xi||_{\infty} ||\nabla \varphi||_{\infty}  L_{\beta} ||\uhnr ||_{L^1\left(\Omega;L^1(0,T;L^1(\Lambda	))\right)}\xrightarrow[m\rightarrow +\infty]{}  0.
\end{eqnarray*}
\noindent Gathering all the previous convergence results, we pass to the limit in \eqref{discretequforlimit} to get that $\mathds{P}$-a.s. in $\Omega$, for all $\xi\in \big\{\phi\in\mathscr{D}(\re):  \phi(T)=0\big\}$ and all $\varphi\in \big\{\psi\in\mathscr{D}(\re^d) : \nabla\psi\cdot\mathbf{n}=0\text{ on }\partial\Lambda\big\},$
\begin{align}\label{210830_08}
\begin{split}
&-\int_0^T\int_\Lambda\left(u(t,x)-\int_0^{t}g_u(s,x)\,dW(s)\right)\xi'(t)\varphi(x)\,dx\,dt-\int_\Lambda u_0(x)\xi(0)\varphi(x)\,dx\\
&-\int_0^T\int_{\Lambda}\big(\ve(t,x)u(t,x)\big)\cdot\nabla \varphi(x)\xi(t)\,dx\,dt\\
=&\int_0^T\int_\Lambda  u(t,x)\Delta\varphi(x)\xi(t)\,dx\,dt+\int_0^T\int_\Lambda \beta_u(t,x)\varphi(x)\xi(t)\,dx\,dt
\end{split}
\end{align}
which can be rewritten as 
\begin{align}\label{210830_08bis}
\begin{split}
&-\int_0^T\int_\Lambda\left(u(t,x)-\int_0^{t}g_u(s,x)\,dW(s)\right)\xi'(t)\varphi(x)\,dx\,dt-\int_\Lambda u_0(x)\xi(0)\varphi(x)\,dx\\
&+\int_0^T\int_{\Lambda}\diver\big(\ve(t,x)u(t,x)\big) \varphi(x)\xi(t)\,dx\,dt\\
=&-\int_0^T \int_\Lambda \nabla u(t,x) \cdot \nabla \varphi(x)\xi(t)\,dx\,dt+\int_0^T\int_\Lambda \beta_u(t,x)\varphi(x)\xi(t)\,dx\,dt,
\end{split}
\end{align}
since $\ve(t,x)\cdot\mathbf{n}(x)=0$ and $\nabla \varphi(x)\cdot\mathbf{n}(x)=0$ for any $(t,x)\in [0,T]\times \partial \Lambda$.
\noindent By \cite[Theorem 1.1]{Droniou} the set $\{\psi\in\mathscr{D}(\mathbb{R}^d)\ | \ \nabla\psi\cdot\mathbf{n}=0 \ \text{on} \ \partial\Lambda\}$ is dense in $H^1(\Lambda)$ and therefore \eqref{210830_08bis} applies to all $\varphi\in H^1(\Lambda)$. 
In the following, we denote the dual space of $H^1(\Lambda)$ by $H^1(\Lambda)^{\ast}$. Recall that we have the following continuous and dense embeddings
\[H^1(\Lambda)\hookrightarrow L^2(\Lambda)\hookrightarrow H^1(\Lambda)^{\ast}.\]
Let us denote the $H^1(\Lambda)$-$H^1(\Lambda)^{\ast}$ duality bracket by $\langle\cdot,\cdot\rangle$ and the $L^2(\Lambda)$ scalar product by $(\cdot,\cdot)$. From Proposition \ref{addreg u}, we know that the weak limit $u$ belongs to $L^2(\Omega;L^2(0,T;H^1(\Lambda)))$ and thus it follows that 
\[\Delta u\in L^2\big(\Omega;L^2(0,T;H^1(\Lambda)^{\ast})\big)\]
which yields
\begin{align}\label{220125_01}
\begin{split}
-\int_0^T\int_{\Lambda}\nabla u(t,x)\cdot\nabla\varphi(x)\xi(t)\,dx\,dt=\int_0^T\langle \Delta u(t,\cdot),\varphi\rangle\xi(t)\,dt
\end{split}
\end{align}
$\mathds{P}$-a.s. in $\Omega$, for all $\xi\in \big\{\phi\in\mathscr{D}(\re):  \phi(T)=0\big\}$ and all 
$\varphi\in H^1(\Lambda)$. Combining \eqref{210830_08bis} with \eqref{220125_01} and with the identity
\begin{align}\label{eqliminitialdata}
-\int_{\Lambda}u_0(x)\varphi(x)\xi(0)\,dx=\int_0^T\int_{\Lambda}u_0(x)\varphi(x)\xi'(t)\,dx\,dt
\end{align}
(see, \cite[Lemma 7.3]{Roubicek}), from Fubini's theorem it follows that
\begin{align*}
&\left\langle-\int_0^T\left(u(t)-\int_0^t g_u(s)\,dW(s)-u_0\right)\xi'(t)\,dt,\varphi\right\rangle\\
&=\left\langle\int_0^T\Big(\Delta u(t)-\diver\big(\ve(t,\cdot)u(t)\big)+\beta_u(t)\Big)\xi(t)\,dt,\varphi\right\rangle,
\end{align*}
$\mathds{P}$-a.s. in $\Omega$, for all $\xi\in \big\{\phi\in\mathscr{D}(\re):  \phi(T)=0\big\}$ and all $\varphi\in H^1(\Lambda)$. Therefore
$$
-\hspace*{-0.1cm}\int_0^T\hspace*{-0.1cm}\left(u(t)-\int_0^t g_u(s)\,dW(s)-u_0\right)\xi'(t)\,dt=\hspace*{-0.1cm}\int_0^T\hspace*{-0.1cm}\Big(\Delta u(t)-\diver\big(\ve(t,\cdot)u(t)\big)+\beta_u(t)\Big)\xi(t)\,dt
$$
in $H^1(\Lambda)^{\ast}$, for all $\xi\in \big\{\phi\in\mathscr{D}(\re):  \phi(T)=0\big\}$, $\mathds{P}$-a.s. in $\Omega$ since, by a separability argument, the exceptional set in $\Omega$ may be chosen independently of $\varphi$. Consequently, (see, e.g. \cite[Proposition A6]{Brezis})
\[\left(u-\int_0^\cdot g_u\,dW-u_0\right)\in W^{1,2}(0,T;H^1(\Lambda)^{\ast})\quad \text{$\mathds{P}$-a.s. in $\Omega$}\]
and so
\begin{align}\label{220126_02}
\hspace*{-0.29cm}\frac{d}{dt}\left(u(t)-\int_0^t g_u(s)\,dW(s)-u_0\right)=\Delta u-\diver(\ve u)+\beta_u \text{ in }  L^2(\Omega;L^2(0,T;H^1(\Lambda)^{\ast}) .
\end{align}
Since $g_u\in L^2_{\mathcal{P}_T}\big(\Omega\times(0,T);H^1(\Lambda)\big)$, it follows thanks to the properties of the stochastic integral that (see \cite{DPZ14} Proposition 4.15 p.104 or \cite{PrevotRockner} Lemma 2.4.1 p.35)
\[\nabla\left(\int_0^\cdot g_u\,dW\right)=\int_0^\cdot\nabla g_u\,dW,\]
hence $\displaystyle u-\int_0^{\cdot}g_u\,dW$ is an element of $L^2(\Omega;L^2(0,T;H^1(\Lambda)))$.
Note that the remaining of the proof is very similar to the one exposed in our previous paper \cite{BNSZ22}, but for the sake of completeness, we decide to detail it again.
From \cite[Lemma 7.3]{Roubicek} we obtain at first that $u\in L^2\big(\Omega;\mathscr{C}([0,T];L^2(\Lambda))\big)$ and together with \eqref{220126_02}, the following rule of partial integration for all $0\leq t\leq T$, $\mathds{P}$-a.s. in $\Omega$: 
\begin{align}\label{220126_01}
\begin{split}
&\left(u(t)-\int_0^t g_u(\tau)\,dW(\tau)-u_0,\zeta(t)\right)-(u(0)-u_0,\zeta(0))\\
=&\int_0^t\left\langle\Delta u(s)-\diver\big(\ve(s,\cdot)u(s)\big) +\beta_u(s),\zeta(s)\right\rangle\,ds\\
&+\int_0^t\left\langle \partial_t\zeta(s),u(s)-\int_0^s g_u(\tau)\,dW(\tau)-u_0\right\rangle\,ds
\end{split}
\end{align}
for all $\zeta\in L^2(0,T;H^1(\Lambda))$ with $\partial_t\zeta\in L^2(0,T;H^1(\Lambda)^{\ast}))$. Choosing $\zeta:(t,x)\mapsto\xi(t)\varphi(x)$ with $\varphi\in H^1(\Lambda)$, $\xi\in \big\{\phi\in\mathscr{D}(\re):  \phi(T)=0\big\}$ in \eqref{220126_01}, we get
\begin{align}\label{220126_03}
\begin{split}
&\left(u(t)-\int_0^t g_u(\tau)\,dW(\tau)-u_0,\varphi\right)\xi(t)-(u(0)-u_0,\varphi)\xi(0)\\
=&\int_0^t\xi(s)\left\langle\Delta u(s)-\diver\big(\ve(s,\cdot)u(s)\big)+\beta_u(s),\varphi\right\rangle\,ds\\
&+\int_0^t \xi'(s)\left(u(s)-\int_0^s g_u(\tau)\,dW(\tau)-u_0,\varphi\right)\,ds
\end{split}
\end{align}
$\mathds{P}$-a.s. in $\Omega$. The particular choice of $t=T$ and $\xi\in\mathscr{D}(\mathbb{R})$ with $\xi(T)=0$ and $\xi(0)=1$ in \eqref{220126_03} combined with \eqref{210830_08}, \eqref{220125_01} and \eqref{eqliminitialdata} yields
\[(u(0)-u_0,\varphi)=0 \quad \text{for all $\varphi\in H^1(\Lambda)$, $\mathds{P}$-a.s. in $\Omega$}\]
and therefore $u(0)=u_0$ $\mathds{P}$-a.s. in $\Omega$.\\
Now, we fix $t\in [0,T)$ and choose $\xi\in\mathscr{D}(\mathbb{R})$ with $\xi(T)=0$ and $\xi(s)=1$ for all $s\in [0,t]$. With this choice, from \eqref{220126_03} we obtain
\begin{align}\label{220127_01}
\hspace*{-0.3cm}\left(u(t)-\int_0^t g_u(s)\,dW(s)-u(0),\varphi\right)\hspace*{-0.1cm}=\hspace*{-0.1cm}\int_0^t\left\langle\Delta u(s)-\diver\big(\ve(s,\cdot)u(s)\big)+\beta_u(s),\varphi\right\rangle\,ds
\end{align}
$\mathds{P}$-a.s. in $\Omega$ for all $\varphi\in H^1(\Lambda)$. Since, for fixed $\varphi\in H^1(\Lambda)$, 
$$t\mapsto \left(u(t)-\int_0^t g_u(s)\,dW(s)-u(0),\varphi\right)$$ 
$$\text{and }  t\mapsto \int_0^t\left\langle\Delta u(s)-\diver\big(\ve(s,\cdot)u(s)\big)+\beta_u(s),\varphi\right\rangle\,ds$$
are continuous on $[0,T]$, $\mathds{P}$-a.s. in $\Omega$, the exceptional set in $\Omega$ in \eqref{220127_01} may be chosen independently of $t\in [0,T)$ and \eqref{220127_01} also holds for $t=T$. This yields
\begin{align*}
u(t)-u(0)-\int_0^t g_u(s)\,dW(s)+\int_0^t \diver\big(\ve(s,\cdot)u(s)\big)ds- \int_0^t\beta_u(s)\,ds=\int_0^t\Delta u(s)\,ds,
\end{align*}
\text{in $H^1(\Lambda)^{\ast}$ and $\mathds{P}$-a.s. in $\Omega$}. 
To conclude, let us mention that since the left-hand side of the above equality is in $L^2(\Lambda)$, it also holds in $L^2(\Lambda)$.
\end{proof}

\begin{lem}{(Stochastic energy equality)
}\label{lemenergy}
For any $c>0$, the stochastic process $u$ introduced in Proposition \ref{addreg u} satisfies the following stochastic energy equality: 
\begin{align}\label{energy}
\begin{split}
&e^{-ct} \E\left[||u(t) ||^2_{L^2(\Lambda)}\right]+2\int_0^te^{-cs} \E\left[||\nabla u(s) ||^2_{L^2(\Lambda)}\right]\,ds\\
=\,& \E\left[|| u_0||^2_{L^2(\Lambda)}\right]-c\int_0^t e^{-cs} \E\left[|| u(s)||^2_{L^2(\Lambda)}\right]\,ds+\int_0^t e^{-cs} \E\left[||g_u(s)||^2_{L^2(\Lambda)}\right]\,ds\\
&+2\int_{0}^te^{-cs} \E\left[ \int_{\Lambda}\beta_u(s,x)u(s,x)\,dx \right]\,ds, \quad \forall t\in [0,T].
\end{split}
\end{align}
\end{lem}
\begin{proof} It is a direct application of It\^o formula to the stochastic process $u$ and the functional $\Psi: (t,v)\mapsto e^{-ct}||v||^2_{L^2(\Lambda)}$ defined on $[0,T]\times L^2(\Lambda)$. Le us precise that in the application of It\^o formula , 
the following contribution of the flux term appears
$$2\int_0^te^{-cs}\E\left[\int_\Lambda \big(\ve(s,x)u(s,x)\big)\cdot \nabla u(s,x)\,dx\right]ds$$
and since $\ve$ is divergence-free and satisfies $\ve\cdot \mathbf{ n}=0$ on $[0,T]\times \partial\Lambda$, one gets 
\begin{eqnarray*}
&&\int_0^te^{-cs}\E\left[\int_\Lambda \big(\ve(s,x)u(s,x)\big)\cdot \nabla u(s,x)\,dx\right]\,ds\\
&=&\int_0^te^{-cs}\E\left[\int_\Lambda \ve(s,x)\cdot \nabla \left (\int_0^{u(s,x)}zdz\right)\,dx\right]\,ds\\
&=&0.
\end{eqnarray*}
\end{proof}

\subsection{Identification of weak limits coming from the non-linear terms}

The next result gives a lower bound on the inferior limit of the integral over the interval $[0,T]$ of an exponential weight in time norm of $\E[|u^r_{h_m,N_m}(\cdot)|^2_{1,h}]$. This result is needed for the identification of the weak limits $g_u$ and $\beta_u$ of respectively $(g(u^l_{h,N}))_m$ and $(\beta(u^r_{h,N}))_m$. Indeed, in such an identification procedure, we are led to make appear the term $\int_0^T\int_0^te^{-cs} \E\left[\int_{\Lambda}|\nabla u(s,x) |^2\,dx\right]\,ds\,dt$ (with a constant $c>0$), and the following lemma is essential for that.

\begin{lem}\label{keylemma} For any $c>0$, the stochastic process $u$ introduced in Proposition \ref{addreg u} satisfies the following inequality:
\begin{eqnarray}\label{liminfinequality}
\hspace*{-0.4cm}\int_0^T\int_0^t e^{-cs}\E\left[\int_{\Lambda}|\nabla u(x,s) |^2\,dx\right]\,ds\,dt\leq\liminf_{m\rightarrow+\infty}\int_0^T\int_{0}^{t}e^{-cs}\E[|u^r_{h_m,N_m}(s)|^2_{1,h_m}]\,ds\,dt.
\end{eqnarray}
\end{lem}
\begin{proof} The idea is to generalize the Lemma 2.2 of \cite{HerbinMarchand} to the evolutionary in time and stochastic case.
Following their work, we start by introducing an approximation of $u$ constructed from several density results. For this we will need the following functional spaces: 
\begin{eqnarray*}
\mathscr{D}(\overline{\Lambda})&=&\left\{\varphi_{|_\Lambda}, \varphi\in \mathscr{C}^\infty_c(\mathbb{R}^d)\right\}\\
\mathcal{V}&=&\left\{\varphi\in \mathscr{D}(\overline{\Lambda}) : \nabla \varphi\cdot\textbf{n}=0\text{ on }\partial \Lambda\right\}\\
\mathscr{V}&=&\left\{\sum_{k \text{ finite}}\xi_k\varphi_k, \xi_k\in \mathscr{C}^\infty_c (0,T) \text{ and }\varphi_k\in\mathcal{V} \right\}.
\end{eqnarray*} Firstly, note that from \cite{ABM} (Chapter 6, Remark 6.2.1.p.223), $\mathcal{V}$ is a dense part of $H^1(\Lambda)$, and secondly, from \cite{D} (Corollary 1.3.1 p.13), $\mathscr{V}$ is dense in $L^2(0,T;H^1(\Lambda))$.
Now owing to \cite{D} (Proposition 1.3.1 p.13), since $u$ belongs to $L^2(\Omega; L^2(0,T;H^1(\Lambda)))$, then there exists a sequence $(u_p)_{p\in \mathbb{N}^{\star}}$ of simple random variables $\displaystyle u_p=\sum_{i=1}^p\mathds{1}_{A_i}\vartheta_i$, where for any $i$ in $\{1,...,p\}$, $A_i\in \mathcal{A}$, $\vartheta_i\in L^2(0,T;H^1(\Lambda))$ and such that 
$$u_p\xrightarrow[p\rightarrow+\infty]{}u\text{ in } L^2(\Omega; L^2(0,T;H^1(\Lambda))).$$ 
Set $p\in \mathbb{N}^{\star}$, then for $i\in \{1,...,p\}$ since $\vartheta_i\in L^2(0,T;H^1(\Lambda))$ there exists a sequence $(\vartheta_{i,\epsilon_i})_{\epsilon_i>0}\subset \mathscr{V}$ such that 
$$\vartheta_{i,\epsilon_i}\xrightarrow[\epsilon_i\rightarrow 0]{}\vartheta_i \text{ in } L^2(0,T;H^1(\Lambda)).$$
Note that since $p$ is finite, we can introduce a common $\epsilon>0$ such that 
$$\vartheta_{i,\epsilon}\xrightarrow[\epsilon\rightarrow 0]{}\vartheta_i \text{ in } L^2(0,T;H^1(\Lambda)).$$
Then we define $\displaystyle \upe=\sum_{i=1}^p \mathds{1}_{A_i}\vartheta_{i,\epsilon}$ and we have for any fixed $p$, $$\upe\xrightarrow[\epsilon\rightarrow 0]{}u_p\text{ in } L^2(\Omega; L^2(0,T;H^1(\Lambda))).$$
Set $c>0$, $p\in  \mathbb{N}^{\star}$ and $\epsilon>0$. We denote $\upenk=\upe(\tn, x_K)$ $\mathds{P}$-a.s. in $\Omega$, for any $n\in \{0,...,N-1\}$ and any $K\in \Tau$. Due to the weak convergence of the finite-volume approximation $(u^r_{h,N})_m$ towards $u$ in $L^2(\Omega; L^2(0,T;L^2(\Lambda)))$, we have
\begin{align}\label{LienT1m}
\begin{split}
&\E\left[\int_0^T\emct\int_{\Lambda} \nabla u(t,x)\cdot \nabla \upe(t,x)\,dx\,dt \right]\\
=\,&-\E\left[\int_0^T\emct\int_{\Lambda} u(t,x)  \diver\left(\nabla \upe(t,x)\right)\,dx\,dt \right]\\
=\,&-\lim_{m\rightarrow +\infty} \E\left[\int_0^T\emct\int_{\Lambda} u^r_{h,N}(t,x)\diver\left(\nabla \upe(t,x)\right)\,dx\,dt \right].
\end{split}
\end{align}
Note that
\begin{eqnarray*}
&&\E\left[\int_0^T\emct\int_{\Lambda} u^r_{h,N}(t,x)\diver\left(\nabla \upe(t,x)\right)\,dx\,dt \right]\\
&=&\E\left[\sum_{n=0}^{N-1}\sum_{K\in \Tau}\int_{\tn}^{\tnp}\emct\int_Ku^{n+1}_K\diver\left(\nabla \upe(t,x)\right)\,dx\,dt \right]\\
&=&\E\left[\sum_{n=0}^{N-1}\sum_{K\in \Tau}u^{n+1}_K\int_{\tn}^{\tnp}\emct \int_{\partial K}\nabla \upe(t,x)\cdot\textbf{n}_{K}(x)\,d\gamma(x)\,dt \right]\\
&=&\E\left[\sum_{n=0}^{N-1}\sum_{\sigma=K|L\in \edgesint}(u^{n+1}_K-u^{n+1}_L)\int_{\tn}^{\tnp}\emct\int_{\sigma}\nabla \upe(t,x)\cdot\textbf{n}_{K,L}\,d\gamma(x)\,dt \right]\\
&=&T_{1,m}-T_{2,m}+T_{2,m}
\end{eqnarray*}
where
\begin{eqnarray*}
T_{1,m}&=&\E\left[\sum_{n=0}^{N-1}\sum_{\sigma=K|L\in \edgesint}(u^{n+1}_K-u^{n+1}_L)\int_{\tn}^{\tnp}\emct\int_{\sigma}\nabla \upe(t,x)\cdot\textbf{n}_{K,L}\,d\gamma(x)\,dt \right]\\
\text{and }T_{2,m}&=&\E\left[\sum_{n=0}^{N-1}\sum_{\sigma=K|L\in \edgesint}(u^{n+1}_K-u^{n+1}_L)\int_{\tn}^{\tnp}\emct\frac{m_{\sigma}}{d_{K|L}}(\upenpl-\upenpk) \,dt\right].
\end{eqnarray*}
Firstly, one shows that $|T_{1,m}-T_{2,m}|\xrightarrow[m\rightarrow +\infty]{}0$. To show this, we rewrite $T_{1,m}-T_{2,m}$ as
\begin{eqnarray*}
&&T_{1,m}-T_{2,m}
=\E\left[\sum_{n=0}^{N-1}\sum_{\sigma=K|L\in \edgesint}\Delta t m_{\sigma}(u^{n+1}_K-u^{n+1}_L)\times R^{n+1}_{\sigma}(\upe) \right]
\end{eqnarray*}
where
\begin{eqnarray}\label{Rnpun}
R^{n+1}_{\sigma}(\upe)=\frac{1}{\Delta t m_{\sigma}}\int_{\tn}^{\tnp}\emct\int_{\sigma}\left(\nabla \upe(t,x)\cdot\textbf{n}_{K,L}-\frac{\upenpl-\upenpk}{d_{K|L}}\right)\,d\gamma(x)\,dt
\end{eqnarray}
which can be rewritten as
\begin{eqnarray*}
R^{n+1}_{\sigma}(\upe)&=&\frac{1}{\Delta t m_{\sigma}}\int_{\tn}^{\tnp}\hspace*{-0.2cm}e^{-ct}\hspace*{-0.1cm}\int_{\sigma}\left(\nabla \upe(t,x)\cdot\textbf{n}_{K,L}-\frac{\upe(t,x_L)-\upe(t,x_K)}{d_{K|L}}\right)\,d\gamma(x)\,dt\\
&&+\frac{1}{\Delta t }\int_{\tn}^{\tnp}e^{-ct}\left(\frac{\upe(t,x_L)-\upe(t,x_K)}{d_{K|L}}-\frac{\upenpl-\upenpk}{d_{K|L}}\right)\,dt.
\end{eqnarray*}
Let us denote by $\mathcal{H}_{\upe}$ the Hessian matrix of $\upe$ defined for any $i,j\in \{1,..,d\}$ by $\displaystyle(\mathcal{H}_{\upe})_{ij}=\frac{\partial^2 \upe}{\partial x_i\partial x_j}$. Then, 
thanks to a Taylor's expansion, for any $t\in [0,T]$ and any $x\in\sigma$, we have \begin{eqnarray*}
\upe(t,x_L)&=&\upe(t,x)+(x_L-x)\cdot\nabla \upe(t,x)+ \frac{1}{2}(x_L-x)^T\mathcal{H}_{\upe}(t,\overline{x})(x_L-x)\\
\upe(t,x_K)&=&\upe(t,x)+(x_K-x)\cdot\nabla \upe(t,x)+ \frac{1}{2}(x_K-x)^T\mathcal{H}_{\upe}(t,\hat{x})(x_K-x),
\end{eqnarray*}
where $\overline{x}=\lambda x_L+(1-\lambda)x$ and $\hat{x}=\alpha x_K+(1-\alpha)x$ for $\lambda,\alpha\in [0,1]$. Remind that the orthogonality condition on the mesh implies that for any $\sigma=K|L\in \edgesint$, $x_L-x_K=d_{K|L}\mathbf{n}_{K,\sigma}$ and so, using \eqref{hoverdkl}
\begin{align}\label{TExpx}
\begin{split}
&\left|\nabla \upe(t,x)\cdot\textbf{n}_{K,L}-\frac{\upe(t,x_L)-\upe(t,x_K)}{d_{K|L}}\right| \\
=\,&\left|\frac{1}{2}\frac{(x_K-x)^T\mathcal{H}_{\upe}(t,\hat{x})(x_K-x)-(x_L-x)^T\mathcal{H}_{\upe}(t,\overline{x})(x_L-x)}{d_{K|L}}\right|\\
\leq\,&h\reg\sup_{(t,y)\in [0,T]\times \Lambda}|\mathcal{H}_{\upe}(t,y)|.
\end{split}
\end{align}
Additionally, using the following equalities for any $t$ in $[\tn,\tnp]$
$$\upenpk=\upe(\tnp,x_K)=\upe(t,x_K)+\int_{t}^{\tnp}\partial_t \upe(s,x_K)\,ds$$
and 
$$\upenpl=\upe(\tnp,x_L)=\upe(t,x_L)+\int_{t}^{\tnp}\partial_t \upe(s,x_L)\,ds,$$
and the fact that for any $s$ in $[0,T]$
\begin{eqnarray*}
\big|\partial_t \upe(s,x_K)-\partial_t \upe(s,x_L)\big|&=&\left|\int_0^1\Big(\nabla\big(\partial_t \upe\big)(s,\mu x_K+(1-\mu) x_L)\Big)\cdot(x_K-x_L)\,d\mu\right|\\
&\leq& 2h \sup_{(s,x)\in [0,T]\times \Lambda}|\nabla(\partial_t \upe)(s,x)|,
\end{eqnarray*}
we obtain using again \eqref{hoverdkl}
\begin{align}\label{TExpt}
\begin{split}
&\left|\frac{\upe(t,x_L)-\upe(t,x_K)}{d_{K|L}}-\frac{\upenpl-\upenpk}{d_{K|L}}\right|\\
=\,&\left|\int_{t}^{\tnp}\frac{\partial_t \upe(s,x_K)-\partial_t \upe(s,x_L)}{\dkl}\,ds\right|\\
\leq\,&\Delta t\reg \hspace*{-0.2cm} \sup_{(s,x)\in [0,T]\times \Lambda}\hspace*{-0.15cm}|\nabla(\partial_t \upe)(s,x)|.
\end{split}
\end{align}
Noticing that $\displaystyle \int_{\tn}^{\tnp}\emct \,dt\leq \dlt$ and combining \eqref{TExpx} and \eqref{TExpt}, we obtain the existence of a constant $K_{p,\epsilon}>0$ only depending on $\upe$ such that
\begin{eqnarray}\label{bornekpe}
\E\left[|R^{n+1}_{\sigma}(\upe)|^2\right]&\leq& \big(K_{p,\epsilon}(\Delta t+h)\reg\big)^2.
\end{eqnarray}
Applying Cauchy-Schwarz inequality yields
\begin{align}\label{majt1m}
\begin{split}
&|T_{1,m}-T_{2,m}|^2\\
=\,&\left|\E\left[\sum_{n=0}^{N-1}\sum_{\sigma=K|L\in \edgesint}\Delta t \frac{m_{\sigma}}{\sqrt{d_{K|L}}}(u^{n+1}_K-u^{n+1}_L)\times \sqrt{d_{K|L}}R^{n+1}_{\sigma}(\upe) \right]\right|^2\\
\leq\,&\E\left[\sum_{n=0}^{N-1}\sum_{\sigma=K|L\in \edgesint}\Delta t \frac{m_{\sigma}}{d_{K|L}}\big|u^{n+1}_K-u^{n+1}_L\big|^2\right]\\
&\times d\E\left[\sum_{n=0}^{N-1}\sum_{\sigma=K|L\in \edgesint}\Delta t \frac{m_{\sigma}d_{K|L}}{d} |R^{n+1}_{\sigma}(\upe) |^2\right]\\
\leq\,& dT|\Lambda|\big(K_{p,\epsilon}(\Delta t+h)\reg\big)^2 \E\left[\int_{0}^T|u^r_{h,N}(t)|_{1,h}^2\,dt\right], 
\end{split}
\end{align}
which tends to $0$ as $m\rightarrow +\infty$. We can thus affirm that $T_{2,m}$ admits a limit which is known thanks to \eqref{LienT1m}
\begin{eqnarray*}
\lim_{m\rightarrow +\infty}T_{2,m}=\lim_{m\rightarrow +\infty}T_{1,m}
=\E\left[\int_0^T\emct\int_{\Lambda}  \nabla u(t,x)\cdot\nabla \upe(t,x)\,dx\,dt \right].
\end{eqnarray*}

Secondly, thanks to Cauchy-Schwarz inequality, we have
\begin{align}\label{majt2m}
|T_{2,m}|^2\leq\,&\E\left[\sum_{n=0}^{N-1}\sum_{\sigma=K|L\in \edgesint}\int_{\tn}^{\tnp}\emct\frac{m_{\sigma}}{d_{K|L}}\big|u^{n+1}_K-u^{n+1}_L\big|^2\,dt\right] \nonumber \\
&\times \E\left[\sum_{n=0}^{N-1}\sum_{\sigma=K|L\in \edgesint}\int_{\tn}^{\tnp}\emct\frac{m_{\sigma}}{d_{K|L}}\big|\upenpk-\upenpl\big|^2\,dt\right]\\
=\,&\E\left[\int_{0}^T\hspace*{-0.1cm}\emct|u^r_{h,N}(t)|_{1,h}^2\,dt\right]\times  \E\left[\sum_{n=0}^{N-1}\sum_{\sigma=K|L\in \edgesint}\int_{\tn}^{\tnp}\hspace*{-0.2cm}\emct\frac{m_{\sigma}}{d_{K|L}}\big|\upenpk-\upenpl\big|^2 dt\right].\nonumber
\end{align}

Now, due to the regularity of $\upe$, by adapting previous arguments applied to $\uhnr$, one can show that 
\begin{eqnarray*}
&&\lim_{m\rightarrow  +\infty}\E\left[\sum_{n=0}^{N-1}\sum_{\sigma=K|L\in \edgesint}\int_{\tn}^{\tnp}\emct\frac{m_{\sigma}}{d_{K|L}}\big|\upenpk-\upenpl\big|^2\,dt\right]\\
&=&\E\left[\int_0^T\emct\int_{\Lambda}|\nabla \upe(t,x)|^2\,dx\,dt\right].
\end{eqnarray*}

Using this convergence result, one obtains by passing to the inferior limit in \eqref{majt2m}

\begin{eqnarray*}
&&\left(\E\left[\int_0^T\emct\int_{\Lambda} \nabla u(t,x)\cdot\nabla \upe(t,x)\,dx\,dt \right]\right)^2\\
&\leq&\liminf_{m\rightarrow +\infty}\E\left[\int_{0}^T\emct|u^r_{h,N}(t)|_{1,h}^2\,dt\right]\times \E\left[\int_0^T\emct\int_{\Lambda}|\nabla \upe(t,x)|^2\,dx\,dt\right].
\end{eqnarray*}
Then, using strong convergences in $L^2(\Omega; L^2(0,T;H^1(\Lambda)))$ of $(\upe)_{\epsilon}$ towards $u_p$ (for a fixed $p$) as $\epsilon$ goes to $0$ and of $(u_p)_p$ towards $u$ as $p$ goes to $+\infty$, the following holds
\begin{eqnarray*}
\E\left[\int_0^T\emct\int_{\Lambda}|\nabla u(t,x)| ^2\,dx\,dt\right]
&\leq&\liminf_{m\rightarrow +\infty}\E\left[\int_{0}^T\emct|u^r_{h,N}(t)|_{1,h}^2\,dt\right]. \\
\end{eqnarray*}
Note that using the same reasoning we show that for any $t\in [0,T]$, 
\begin{eqnarray*}
\E\left[\int_0^t\emcs\int_{\Lambda}|\nabla u(s,x)|^2\,dx\,ds\right]
&\leq&\liminf_{m\rightarrow +\infty}\E\left[\int_{0}^t\emcs|u^r_{h,N}(s)|_{1,h}^2\,ds\right] \\
\end{eqnarray*}
and thanks to Fatou's Lemma, we finally have
\begin{eqnarray*}
\int_{0}^T\int_0^t\emcs\E\left[\int_{\Lambda}|\nabla u(x,s)| ^2\,dx\right]\,ds\,dt
&=&\int_{0}^T\E\left[\int_0^t\emcs\int_{\Lambda}|\nabla u(s,x)| ^2\,dx\,ds\right]\,dt\\
&\leq&\int_{0}^T\liminf_{m\rightarrow +\infty}\E\left[\int_{0}^t\emcs|u^r_{h,N}(s)|_{1,h}^2\,ds\right]\,dt \\
&\leq&\liminf_{m\rightarrow +\infty}\int_{0}^T\E\left[\int_{0}^t\emcs|u^r_{h,N}(s)|_{1,h}^2\,ds\right]\,dt\\
&=&\liminf_{m\rightarrow +\infty}\int_{0}^T\int_{0}^t\emcs\E\left[|u^r_{h,N}(s)|_{1,h}^2\right]\,ds\,dt,
\end{eqnarray*}
and the proof is complete.
\end{proof}

\begin{remark} Let us detail here why we do consider such an approximation of $u$ in the previous proof. Firstly, we need to have the following equality: 
$$\E\left[\int_0^T \emct\hspace*{-0.1cm}\int_{\Lambda}\nabla u(t,x)\cdot \nabla \upe(t,x)\,dx\,dt\right]=-\E\left[\int_0^T\emct\hspace*{-0.1cm} \int_{\Lambda} u(t,x)\diver(\nabla \upe(t,x))\,dx\,dt\right],$$
and if $\nabla \upe\cdot\mathbf{ n}\neq 0$ then the boundary term appears in the application of the Gauss-Green formula and our argumentation fails. For this reason, we choose the density of $\mathcal{V}$ in $H^1(\Lambda)$. Secondly, we need a control of the term $\E\big[|R^{n+1}_{\sigma}(\upe)|^2\big]$, and for this reason we choose a particular approximation $(\upe)_{p,\epsilon}$ of elementary processes's type. Indeed, if we consider another approximation $(U_{p,\epsilon})_{p,\epsilon}$ with the only information $(U_{p,\epsilon})_{p,\epsilon}\subset L^2(\Omega; \mathscr{V})$, we don't know if we control or not the expectation of the following quantities: 
$$\left(\sup_{(t,y)\in [0,T]\times \Lambda}\big|\mathcal{H}_{U_{p,\epsilon}}(t,y)\big|\right)^2\text{ and } \left(\sup_{(s,x)\in [0,T]\times \Lambda}\big|\nabla(\partial_t U_{p,\epsilon})( s,x)\big|\right)^2.$$
\end{remark}

Now, we have all the necessary tools on the one hand for the identification of $g_u$ and $\beta_u$, and on the other hand for completing the proof of Theorem \ref{mainresult}.

\begin{prop}\label{PISL}
The sequences $(u_{h,N}^r)_m$ and $(u_{h,N}^l)_m$ converge strongly in \linebreak $L^2(\Omega;L^2(0,T;L^2(\Lambda)))$ to the unique variational solution of Problem \eqref{equation} in the sense of Definition \ref{solution}.
\end{prop}

\begin{proof}
Let us fix $n\in \{0,...,N-1\}$, $K\in \Tau$, and multiply \eqref{equationapproxter}
by $\Delta t u_K^{n+1}$, use the formula $a(a-b)=\frac{1}{2}(a^2-b^2+(a-b)^2)$ with $a=u_K^{n+1}$ and $b=u_K^n$, take the expectation, and proceed as for the obtention of \eqref{term3} to arrive at
\begin{eqnarray*}
&&\frac{m_K}{2}\E\left[(u_K^{n+1})^2-(u_K^n)^2\right]+\frac{m_K}{2}\E\left[(u_K^{n+1}-u_K^n)^2\right]\\
&&+\dlt\sum_{\sigma=K|L\in\edgesint\cap\edges_K} m_\sigma (\vksnp)^-\erwb(\unpk-\unpl)u_K^{n+1}\erwe\\
&&+\Delta t\sum_{\sigma=K|L\in\edgesint\cap\edges_K}\frac{m_\sigma}{\dkl}\E\left[(u_K^{n+1}-u_L^{n+1})u_K^{n+1}\right]\\
&\leq&\frac{m_K}{2}\E\left[(u_K^{n+1}-u_K^n)^2\right]+\frac{m_K\Delta t}{2}\E\left[g^2(u_K^n)\right]+\dlt m_K \E\left[\beta(\unpk)\unpk\right].
\end{eqnarray*}
Now, we multiply the last inequality by $e^{-c\tn}$ for arbitrary $c>0$. Then, summing over $K\in \Tau$ and $n\in\{0, ..., k\}$ for $k\in\{0, ..., N-1\}$, using \eqref{PInt} and reasoning as in the proof of \eqref{term2} one gets
\begin{eqnarray*}
&&\frac{1}{2}\sum_{n=0}^{k}\sum_{K\in \Tau}m_Ke^{-c\tn}\E\left[(u_K^{n+1})^2-(u_K^n)^2\right]+\Delta t\sum_{n=0}^{k}e^{-c\tn}\hspace*{-0.4cm}\sum_{\sigma=K|L\in\edgesint}\frac{m_\sigma}{\dkl}\E\left[|u_K^{n+1}-u_L^{n+1}|^2\right]\\
&\leq&\frac{\Delta t}{2}\sum_{n=0}^{k}\sum_{K\in \Tau}m_Ke^{-c\tn}\E\left[g^2(u_K^n)\right]+\Delta t\sum_{n=0}^{k}\sum_{K\in \Tau}m_Ke^{-c\tn}\E\left[\beta(\unpk)\unpk\right].
\end{eqnarray*}
Let us focus on each sum of this last inequality separately. \\
$\bullet$ Note that the general term of the first sum can be decomposed in the following way: 
\begin{eqnarray*}
&&e^{-c\tn}\E\left[(u_K^{n+1})^2-(u_K^n)^2\right]\\
&=&e^{-c\tn}\E\left[(u_K^{n+1})^2\right]-e^{-ct_{n-1}}\E\left[(u_K^{n})^2\right]-\E\left[(u^n_K)^2\right]\left( e^{-c\tn}-e^{-ct_{n-1}}  \right),
\end{eqnarray*}
where $t_{-1}:=-\Delta t$.
Firstly, we have
\begin{align}\label{terme1a}
\begin{split}
&\frac{1}{2}\sum_{n=0}^{k}\sum_{K\in \Tau}m_K \left(e^{-c\tn}\E\left[(u_K^{n+1})^2\right]-e^{-ct_{n-1}}\E\left[(u_K^{n})^2\right]\right)\\
=\,&\frac{1}{2}\sum_{K\in \Tau} m_Ke^{-ct_k}\E\left[(u_K^{k+1})^2\right]-\frac{1}{2}\sum_{K\in \Tau} m_K\E\left[(u_K^{0})^2\right]e^{c\Delta t}.
\end{split}
\end{align}
Secondly, since there exists $\xi\in \big(-c\tn, -ct_{n-1}\big)$ such that $$e^{-c\tn}-e^{-ct_{n-1}}=e^\xi (-c\tn+ct_{n-1})=-c\Delta te^\xi<-c\Delta t e^{-c\tn},$$
the following inequality holds
\begin{eqnarray*}
-\sum_{n=1}^{k}\sum_{K\in \Tau}m_K \E\left[(u^n_K)^2\right]\left( e^{-c\tn}-e^{-ct_{n-1}} \right)>c\Delta t \sum_{n=1}^{k}\sum_{K\in \Tau}m_K e^{-c\tn} \E\left[(u_K^{n})^2\right].
\end{eqnarray*}
In particular, we have since for any $s\in [\tn, \tnp]$, $e^{-cs}\leq e^{-c\tn}$,
\begin{eqnarray*}
c\Delta t \sum_{n=1}^{k}\sum_{K\in \Tau}m_K e^{-c\tn} \E\left[(u_K^{n})^2\right]
&=&c\Delta t \sum_{n=0}^{k-1}\sum_{K\in \Tau}m_K e^{-c\tnp} \E\left[(u_K^{n+1})^2\right]\\
&=&ce^{-c\Delta t}\sum_{n=0}^{k-1}  \sum_{K\in \Tau}m_K \Delta te^{-c\tn}\E\left[(u_K^{n+1})^2\right]\\
&\geq&ce^{-c\Delta t}\sum_{n=0}^{k-1}  \sum_{K\in \Tau}m_K\int_{\tn}^{\tnp}e^{-cs}\E\left[(u_K^{n+1})^2\right]ds\\
&=&ce^{-c\Delta t}\int_0^{t_k}e^{-cs}\E\left[||u^r_{h,N}(s)||^2_{L^2(\Lambda)}\right]ds.
\end{eqnarray*}
In this manner
\begin{align}\label{terme1b}
\begin{split}
&-\frac{1}{2}\sum_{n=0}^{k}\sum_{K\in \Tau}m_K \E\left[(u^n_K)^2\right]\left( e^{-c\tn}-e^{-ct_{n-1}}  \right)\\
>&-\frac{1}{2}\sum_{K\in \Tau}m_K \E\left[(u^0_K)^2\right]\left( 1-e^{c\Delta t}  \right) +\frac{c}{2}e^{-c\Delta t}\int_0^{t_k}e^{-cs}\E\left[||u^r_{h,N}(s)||^2_{L^2(\Lambda)}\right]ds\\
>&\ \frac{c}{2}e^{-c\Delta t}\int_0^{t_k}e^{-cs}\E\left[||u^r_{h,N}(s)||^2_{L^2(\Lambda)}\right]ds.
\end{split}
\end{align}
$\bullet$ The second sum can be handled in the following manner
\begin{align}\label{terme2}
\begin{split}
\Delta t\sum_{n=0}^{k}e^{-c\tn}\sum_{\sigma=K|L\in\edgesint}\frac{m_\sigma}{\dkl}\E\left[|u_K^{n+1}-u_L^{n+1}|^2\right]
=&\ \Delta t\sum_{n=0}^{k}e^{-c\tn}\E[|u_h^{n+1}|^2_{1,h}]\\
\geq&\ \int_{0}^{t_{k+1}}e^{-cs}\E[|u^r_{h,N}(s)|^2_{1,h}]\, ds.
\end{split}
\end{align}
Indeed, since for any $s\in [\tn, \tnp]$, $e^{-cs}\leq e^{-c\tn}$ one gets
\begin{eqnarray*}
\int_{0}^{t_{k+1}}e^{-cs}\E[|u^r_{h,N}(s)|^2_{1,h}]\,ds&=&\sum_{n=0}^k\int_{\tn}^{\tnp}e^{-cs}\E[|u_h^{n+1}|^2_{1,h}]\,ds\\
&\leq&\sum_{n=0}^k \E[|u_h^{n+1}|^2_{1,h}]e^{-c\tn}\int_{\tn}^{\tnp}1\,ds\\
&=&\Delta t\sum_{n=0}^{k}e^{-c\tn}\E[|u_h^{n+1}|^2_{1,h}].
\end{eqnarray*}
$\bullet$ We have the following majoration of the third sum: 
\begin{align}\label{terme3}
\begin{split}
&\frac{\Delta t}{2}\sum_{n=0}^{k}\sum_{K\in \Tau}m_Ke^{-c\tn}\E\left[g^2(u_K^n)\right]\\
\leq\,&\frac{\Delta t}{2}\sum_{K\in \Tau}m_K\E\left[g^2(u_K^0)\right]+\frac{1}{2}\int_{0}^{t_k} e^{-cs}\E\left[||g(u^r_{h,N})(s)||^2_{L^2(\Lambda)}\right]\,ds.
\end{split}
\end{align}
Indeed, 
\begin{eqnarray*}
\frac{\Delta t}{2}\sum_{n=1}^{k}\sum_{K\in \Tau}m_Ke^{-c\tn}\E\left[g^2(u_K^n)\right]
&=&\frac{\Delta t}{2}\sum_{n=0}^{k-1}\sum_{K\in \Tau}m_Ke^{-c\tnp}\E\left[g^2(u_K^{n+1})\right]\\
&\leq&\frac{1}{2}\sum_{n=0}^{k-1}\sum_{K\in \Tau}m_K\int_{\tn}^{\tnp }e^{-cs}\E\left[g^2(u_K^{n+1})\right]\,ds\\
&=&\frac{1}{2}\int_{0}^{t_k} e^{-cs}\E\left[||g(u^r_{h,N})(s)||^2_{L^2(\Lambda)}\right]\,ds.
\end{eqnarray*}
$\bullet$ The last sum can be handled in the following manner: 
\begin{align}\label{terme4}
\begin{split}
&\Delta t\sum_{n=0}^{k}\sum_{K\in \Tau}m_Ke^{-c\tn}\E\left[\beta(\unpk)\unpk\right]\\
\leq\,&c\dlt L_{\beta} ||\uhnr||^2_{L^2\left(\Omega;L^2(0,T;L^2(\Lambda))\right)}+\int_{0}^{t_{k+1}}e^{-cs}\E\left[\int_{\Lambda}\beta(\uhnr(s,x))\uhnr(s,x)\,dx\right]\,ds.
\end{split}
\end{align}
Indeed, 
\begin{align*}
& \left|\Delta t\sum_{n=0}^{k}\sum_{K\in \Tau}m_Ke^{-c\tn}\E\left[\beta(\unpk)\unpk\right]-\int_{0}^{t_{k+1}}e^{-cs}\E\left[\int_{\Lambda}\beta(\uhnr(s,x))\uhnr(s,x)\,dx\right]\,ds\right|\\
=& \left|\sum_{n=0}^{k}\sum_{K\in \Tau}\int_{\tn}^{\tnp}\int_{K}(e^{-c\tn}-e^{-cs})\E\left[\beta(\unpk)\unpk\right]\,dx\,ds\right|\leq c\dlt L_{\beta} ||\uhnr||^2_{L^2\left(\Omega;L^2(0,T;L^2(\Lambda))\right)}.
\end{align*}
Combining \eqref{terme1a}, \eqref{terme1b}, \eqref{terme2}, \eqref{terme3} and \eqref{terme4} one gets
\begin{eqnarray*}
&&\sum_{K\in \Tau} m_Ke^{-ct_k}\E\left[(u_K^{k+1})^2\right]+2\int_{0}^{t_{k+1}}e^{-cs}\E[|u^r_{h,N}(s)|^2_{1,h}]\,ds\\
&\leq&e^{c\Delta t}\sum_{K\in \Tau} m_K\E\left[(u_K^{0})^2\right]+\Delta t\sum_{K\in \Tau}m_K\E\left[g^2(u_K^0)\right]+\int_{0}^{t_k} e^{-cs}\E\left[||g(u^r_{h,N})(s)||^2_{L^2(\Lambda)}\right]ds\\
&&-ce^{-c\Delta t}\int_0^{t_k}e^{-cs}\E\left[||u^r_{h,N}(s)||^2_{L^2(\Lambda)}\right]ds +2c\dlt L_{\beta} ||\uhnr||^2_{L^2\left(\Omega;L^2(0,T;L^2(\Lambda))\right)}\\
&&+2\int_{0}^{t_{k+1}}e^{-cs}\E\left[\int_{\Lambda}\beta(\uhnr(s,x))\uhnr(s,x)\,dx\right]\,ds.
\end{eqnarray*}
For $t\in [t_k, t_{k+1})$ since $e^{-ct}\leq e^{-ct_k}$ and $(t-\Delta t)^+\leq t_k$, one gets
\begin{eqnarray*}
&&e^{-ct}\E\left[||u_{h,N}^r(t)||^2_{L^2(\Lambda)}\right]+2\int_{0}^{t}e^{-cs}\E[|u^r_{h,N}(s)|^2_{1,h}]\,ds\\
&\leq&e^{c\Delta t}\sum_{K\in \Tau} m_K\E\left[(u_K^{0})^2\right]+\Delta t\sum_{K\in \Tau}m_K\E\left[g^2(u_K^0)\right]+\int_{0}^{t} e^{-cs}\E\left[||g(u^r_{h,N})(s)||^2_{L^2(\Lambda)}\right]\,ds\\
&&-ce^{-c\Delta t}\int_0^{(t-\Delta t)^+}e^{-cs}\E\left[||u^r_{h,N}(s)||^2_{L^2(\Lambda)}\right]\,ds +2c\dlt L_{\beta} ||\uhnr||^2_{L^2\left(\Omega;L^2(0,T;L^2(\Lambda))\right)}\\
&&+2\int_{0}^{t}e^{-cs}\E\left[\int_{\Lambda}\beta(\uhnr(s,x))\uhnr(s,x)\,dx\right]\,ds\\
&&+2\int_{t}^{t_{k+1}}e^{-cs}\E\left[\int_{\Lambda}\beta(\uhnr(s,x))\uhnr(s,x)\,dx\right]\,ds
\end{eqnarray*}
Using the fact that $(u^r_{h,N})_{m}$ is bounded in $L^{\infty}\left(0,T;L^2(\Omega;L^2(\Lambda))\right)$ by Proposition~\ref{bounds} one gets 
\begin{align*}
&-ce^{-c\Delta t}\int_{(t-\Delta t)^+}^te^{-cs}\E\left[||u^r_{h,N}(s)||^2_{L^2(\Lambda)}\right]ds\\
&+2\int_{t}^{t_{k+1}}e^{-cs}\E\left[\int_{\Lambda}\beta(\uhnr(s,x))\uhnr(s,x)\,dx\right]\,ds\\
\leq\,&\Delta t (c+2L_{\beta})||u^r_{h,N}||^2_{L^{\infty}\left(0,T;L^2(\Omega;L^2(\Lambda))\right)}
\end{align*}
and so (since $-\int_0^{(t-\Delta t)^+}=-\int_0^t+\int_{(t-\Delta t)^+}^t$)
\begin{align}\label{ineginter}
&e^{-ct}\E\left[||u_{h,N}^r(t)||^2_{L^2(\Lambda)}\right]+2\int_{0}^{t}e^{-cs}\E[|u^r_{h,N}(s)|^2_{1,h}]\,ds\nonumber \\
\leq\,&e^{c\Delta t}\E[||u_0||^2_{L^2(\Lambda)}]+\Delta t L_g^2\E[||u_0||^2_{L^2(\Lambda)}]+\int_{0}^{t} e^{-cs}\E\left[||g(u^r_{h,N})(s)||^2_{L^2(\Lambda)}\right]\,ds\\
&-ce^{-c\Delta t}\int_0^{t}e^{-cs}\E\left[||u^r_{h,N}(s)||^2_{L^2(\Lambda)}\right]\,ds+2c\Delta t L_\beta\|\uhnr\|_{L^2(\Omega;L^2(0,T;L^2(\Lambda)))}^2 \nonumber \\
&+\Delta t  (c+2L_{\beta})||u^r_{h,N}||^2_{L^{\infty}\left(0,T;L^2(\Omega;L^2(\Lambda))\right)}+2\int_{0}^{t}e^{-cs}\E\left[\int_{\Lambda}\beta(\uhnr(s,x))\uhnr(s,x)\,dx\right]\,ds.\nonumber
\end{align}
Moreover,
\begin{align}\label{decompog}
\begin{split}
&\int_{0}^{t} e^{-cs}\E\left[||g(u^r_{h,N})(s)||^2_{L^2(\Lambda)}\right]ds=\int_{0}^{t} e^{-cs}\E\left[||g(u^r_{h,N})(s)-g(u)(s)||^2_{L^2(\Lambda)}\right]ds\\
&+2\int_{0}^{t} e^{-cs}\E\left[\int_{\Lambda}g(u^r_{h,N})(s,x)g(u)(s,x)dx\right]\,ds-\int_{0}^{t} e^{-cs}\E\left[||g(u)(s)||^2_{L^2(\Lambda)}\right]ds.
\end{split}
\end{align}
In the same manner,
\begin{eqnarray}\label{decompou}
&&\hspace*{-1.3cm}-ce^{-c\Delta t}\int_0^{t}e^{-cs}\E\left[||u^r_{h,N}(s)||^2_{L^2(\Lambda)}\right]ds=-ce^{-c\Delta t}\int_{0}^{t} e^{-cs}\E\left[||u^r_{h,N}(s)-u(s)||^2_{L^2(\Lambda)}\right]ds\nonumber\\
&&\hspace*{-1.3cm}-2ce^{-c\Delta t}\int_{0}^{t} e^{-cs}\E\left[\int_{\Lambda}u^r_{h,N}(s,x)u(s,x)\,dx\right]\,ds+ce^{-c\Delta t}\int_{0}^{t} e^{-cs}\E\left[||u(s)||^2_{L^2(\Lambda)}\right]ds.
\end{eqnarray}
At last
\begin{align}\label{decompobeta}
&\int_{0}^{t}e^{-cs}\E\left[\int_{\Lambda}\beta(\uhnr(s,x))\uhnr(s,x)\,dx\right]\,ds=\int_{0}^{t}e^{-cs}\E\left[\int_{\Lambda}\beta(\uhnr(s,x))u(s,x)\,dx\right]\,ds\nonumber\\
&+\int_{0}^{t}e^{-cs}\E\left[\int_{\Lambda}\big(\beta(\uhnr(s,x))-\beta(u(s,x))\big)(\uhnr(s,x)-u(s,x))\,dx\right]\,ds\\
&+\int_{0}^{t}e^{-cs}\E\left[\int_{\Lambda}\beta(u(s,x))(\uhnr(s,x)-u(s,x))\,dx\right]\,ds.\nonumber
\end{align}
Note that there exists $c>0$ depending only on $L_g$ and $L_{\beta}$ such that for any $N$ big enough
\begin{align*}
&\int_{0}^{t} e^{-cs}\E\left[||g(u^r_{h,N}(s))-g(u(s))||^2_{L^2(\Lambda)}\right]ds-ce^{-c\Delta t}\int_{0}^{t} e^{-cs}\E\left[||u^r_{h,N}(s)-u(s)||^2_{L^2(\Lambda)}\right]ds\\
&+2\int_{0}^{t}e^{-cs}\E\left[\int_{\Lambda}\big(\beta(\uhnr(s,x))-\beta(u(s,x))\big)(\uhnr(s,x)-u(s,x))\,dx\right]\,ds\leq 0.
\end{align*}
From now on, we will consider such a choice of $c$.
After injecting \eqref{decompog}, \eqref{decompou} and \eqref{decompobeta} in \eqref{ineginter}, we arrive at
\begin{align*}
&e^{-ct}\E\left[||u_{h,N}^r(t)||^2_{L^2(\Lambda)}\right]+2\int_{0}^{t}e^{-cs}\E[|u^r_{h,N}(s)|^2_{1,h}]\,ds\\
\leq\,&\E[||u_0||^2_{L^2(\Lambda)}]+\Delta t L_g^2\E[||u_0||^2_{L^2(\Lambda)}]+2\int_{0}^{t} e^{-cs}\E\left[\int_{\Lambda}g(u^r_{h,N})(s,x)g(u)(s,x)\,dx\right]\,ds\\
&-\int_{0}^{t} e^{-cs}\E\left[||g(u)(s)||^2_{L^2(\Lambda)}\right]ds-2ce^{-c\Delta t}\int_{0}^{t} e^{-cs}\E\left[\int_{\Lambda}u^r_{h,N}(s,x)u(s,x)\,dx\right]\,ds\\
&+ce^{-c\Delta t}\int_{0}^{t} e^{-cs}\E\left[||u(s)||^2_{L^2(\Lambda)}\right]ds+2c\Delta tL_\beta\|\uhnr\|_{L^2(\Omega;L^2(0,T);L^2(\Lambda)))}^2\\
&+\Delta t (c+2L_{\beta}) ||u^r_{h,N}||^2_{L^{\infty}\left(0,T;L^2(\Omega;L^2(\Lambda))\right)}+\left(e^{c\Delta t}-1\right)\E\left[||u_0||^2_{L^2(\Lambda)}\right]\\
&+2\int_{0}^{t}e^{-cs}\E\left[\int_{\Lambda}\beta(\uhnr(s,x))u(s,x)\,dx\right]\,ds\\
& +2\int_{0}^{t}e^{-cs}\E\left[\int_{\Lambda}\beta(u(s,x))\big(\uhnr(s,x)-u(s,x)\big)\,dx\right]\,ds.
\end{align*}
Then, integrating this last inequality from $0$ to $T$ the following holds
\begin{align*}
&\int_0^Te^{-ct}\E\left[||u_{h,N}^r(t)||^2_{L^2(\Lambda)}\right]\,dt+2\int_0^T\int_{0}^{t}e^{-cs}\E[|u^r_{h,N}(s)|^2_{1,h}]\,ds\,dt\\
\leq\,&\int_0^T\E[||u_0||^2_{L^2(\Lambda)}]\,dt+2\int_0^T\int_{0}^{t} e^{-cs}\E\left[\int_{\Lambda}g(u^r_{h,N})(s,x)g(u)(s,x)\,dx\right]\,ds\,dt\\
&-\int_0^T\int_{0}^{t} e^{-cs}\E\left[||g(u(s))||^2_{L^2(\Lambda)}\right]\,ds\,dt\\
&-2ce^{-c\Delta t}\int_0^T\int_{0}^{t} e^{-cs}\E\left[\int_{\Lambda}u^r_{h,N}(s,x)u(s,x)\,dx\right]\,ds\,dt\\
&+ce^{-c\Delta t}\int_0^T\int_{0}^{t} e^{-cs}\E\left[||u(s)||^2_{L^2(\Lambda)}\right]\,ds\,dt+T\left(e^{c\Delta t}-1\right)\E[||u_0||^2_{L^2(\Lambda)}]\\
&+T\Delta t \Big((c+2L_{\beta})||u^r_{h,N}||^2_{L^{\infty}\left(0,T;L^2(\Omega;L^2(\Lambda))\right)}+L_g^2\E[||u_0||^2_{L^2(\Lambda)}]\\
&+2cL_\beta\|\uhnr\|_{L^2(\Omega;L^2(0,T;L^2(\Lambda)))}^2\Big)\\
&+2\int_{0}^{T}\int_{0}^{t}e^{-cs}\E\left[\int_{\Lambda}\beta(\uhnr(s,x))u(s,x)\,dx\right]\,ds\,dt\\
&+2\int_{0}^{T}\int_{0}^{t}e^{-cs}\E\left[\int_{\Lambda}\beta(u(s,x))\big(\uhnr(s,x)-u(s,x)\big)\,dx\right]\,ds\,dt.
\end{align*}
Firstly, by passing to the superior limit one gets
\begin{align*}
&\limsup_{m\rightarrow+\infty}\int_0^Te^{-ct}\E\left[||u_{h,N}^r(t)||^2_{L^2(\Lambda)}\right]\,dt+2\liminf_{m\rightarrow+\infty}\int_0^T\int_{0}^{t}e^{-cs}\E[|u^r_{h,N}(s)|^2_{1,h}]\,ds\,dt\\
&\leq\int_0^T\E[||u_0||^2_{L^2(\Lambda)}]\,dt-c\int_0^T\int_{0}^{t} e^{-cs}\E\left[||u(s)||^2_{L^2(\Lambda)}\right]\,ds\,dt\\
&+2\int_0^T\int_{0}^{t} e^{-cs}\E\left[\int_{\Lambda}g_u(s,x)g(u)(s,x)\,dx\right]\,ds\,dt-\int_0^T\int_{0}^{t} e^{-cs}\E\left[||g(u(s))||^2_{L^2(\Lambda)}\right]\,ds\,dt\\
&+2\int_{0}^{T}\int_{0}^{t}e^{-cs}\E\left[\int_{\Lambda}\beta_u(s,x)u(s,x)\,dx\right]\,ds\,dt.
\end{align*}
Secondly, thanks to the stochastic energy equality \eqref{energy} one arrives at
\begin{align*}
&\limsup_{m\rightarrow+\infty}\int_0^Te^{-ct}\E\left[||u_{h,N}^r(t)||^2_{L^2(\Lambda)}\right]\,dt+2\liminf_{m\rightarrow+\infty}\int_0^T\int_{0}^{t}e^{-cs}\E[|u^r_{h,N}(s)|^2_{1,h}]\,ds\,dt\\
&\leq\int_0^T    e^{-ct} \E\left[||u(t) ||^2_{L^2(\Lambda)}\right]\,dt +2\int_0^T\int_0^te^{-cs} \E\left[||\nabla u(s) ||^2_{L^2(\Lambda)}\right]\,ds\,dt\\
&-\int_0^T\int_0^t e^{-cs} \E\left[||g_u(s) ||^2_{L^2(\Lambda)}\right]\,ds\,dt+2\int_0^T\int_{0}^{t} e^{-cs}\E\left[\int_{\Lambda}g_u(s,x)g(u(s,x))\,dx\right]\,ds\,dt\\
&-\int_0^T\int_{0}^{t} e^{-cs}\E\left[||g(u(s))||^2_{L^2(\Lambda)}\right]\,ds\,dt,
\end{align*}
which yields to
\begin{eqnarray*}
&&\limsup_{m\rightarrow+\infty}\int_0^Te^{-ct}\E\left[||u_{h,N}^r(t)||^2_{L^2(\Lambda)}\right]\,dt+2\liminf_{m\rightarrow+\infty}\int_0^T\int_{0}^{t}e^{-cs}\E[|u^r_{h,N}(s)|^2_{1,h}]\,ds\,dt\\
&\leq&\int_0^T    e^{-ct} \E\left[||u(t) ||^2_{L^2(\Lambda)}\right]\,dt +2\int_0^T\int_0^te^{-cs} \E\left[||\nabla u(s) ||^2_{L^2(\Lambda)}\right]\,ds\,dt\\
&&-\int_0^T\int_0^t e^{-cs} \E\left[||g(u(s))-g_u(s)||^2_{L^2(\Lambda)}\right]\,ds\,dt.
\end{eqnarray*}
Thirdly, owing to \eqref{liminfinequality} one obtains
\begin{eqnarray*}
&&\limsup_{m\rightarrow+\infty}\int_0^Te^{-ct}\E\left[||u_{h,N}^r(t)||^2_{L^2(\Lambda)}\right]\,dt+\int_0^T\int_0^t e^{-cs} \E\left[||g(u(s))-g_u(s) ||^2_{L^2(\Lambda)}\right]\,ds\,dt
\\
&\leq&\int_0^T    e^{-ct} \E\left[||u(t) ||^2_{L^2(\Lambda)}\right]\,dt.
\end{eqnarray*}
\noindent Note that by weak convergence of $(u_{h,N}^r)_{m}$ towards $u$ in $L^2(\Omega;L^2(0,T;L^2(\Lambda)))$, the following inequality is always true 
$$\int_0^T    e^{-ct} \E\left[||u(t) ||^2_{L^2(\Lambda)}\right]\,dt\leq \liminf_{m\rightarrow+\infty}\int_0^Te^{-ct}\E\left[||u_{h,N}^r(t)||^2_{L^2(\Lambda)}\right]\,dt,$$
and so this allows us to conclude firstly that $g_u=g(u)$, secondly that $(u_{h,N}^r)_m$, $(u_{h,N}^l)_m$ converges strongly to $u$ in $L^2(\Omega;L^2(0,T;L^2(\Lambda)))$ and thirdly that $\beta_u=\beta(u)$. 
At last, thanks to  Vitali's theorem (see \cite{D}), the boundedness  of $(\uhnr)_m$ in $L^{\infty}(0,T; L^2(\Omega\times \Lambda))$ combined with its strong convergence in $L^2(\Omega;L^2(0,T;L^2(\Lambda)))\cong L^{2}(0,T; L^2(\Omega\times \Lambda))$, allow us to conclude that such a convergence finally holds strongly in $L^{p}(0,T; L^2(\Omega\times \Lambda))$ for any finite $p\geq 1$, and then that the limit $u$ is the unique variational solution of Problem \eqref{equation} in the sense of Definition \ref{solution}.
\end{proof}
\quad\\
\textbf{Acknowledgments} The authors would like to thank T. Gallou\"et for
his valuable suggestions. This work has been supported by the German Research Foundation project (ZI 1542/3-1), the Institut de M\'ecanique et d'Ingenierie of Marseille and various Procope programs: Project-Related Personal Exchange France-Germany (49368YE), Procope Mobility Program (DEU-22-0004 LG1) and Procope Plus Project.

%\nocite{*}

\bibliographystyle{plain}
\bibliography{bibfinitevolume2.bib}

\end{document}